\documentclass[12pt]{article}
\usepackage[english]{babel}
\usepackage[T1]{fontenc}

\usepackage{amssymb, amsmath, longtable}
\usepackage{amsthm}
\usepackage{charter,amssymb,amsmath,centernot,amsthm,amsfonts,mathrsfs,latexsym,rotate,epsfig,graphicx,caption,subcaption, longtable}
\usepackage{multirow}
\usepackage{amsthm,array,hyperref,caption,subcaption, float,longtable,multirow}

\newcolumntype{M}[1]{>{\centering\arraybackslash}m{#1}}
\newcolumntype{P}[1]{>{\centering\arraybackslash}p{#1}}

\usepackage[margin=0.9 in]{geometry}
\usepackage{enumerate}
\usepackage{epstopdf} 
\allowdisplaybreaks[1]
\numberwithin{equation}{section}

\newtheorem{example}{Example}[section]
\newtheorem{thm}{Theorem}[section]
\newtheorem{cor}{Corollary}[section]
\newtheorem{note}{Note}[section]
\newtheorem{notation}{Notation}[section]
\newtheorem{notations}{Notations}[section]

\newtheorem{defn}{Definition}[section]

\usepackage{graphicx}
\begin{document}
	\vspace{-2cm}
	\markboth{R. Rajkumar and R. Vishnupriya}{Tensor join of hypergraphs and their spectra}
	\title{\LARGE\bf Tensor join of hypergraphs and its spectra}
	\author{R. Vishnupriya\footnote{e-mail: {\tt rrvmaths@gmail.com}},\ \ \
		R. Rajkumar\footnote{e-mail: {\tt rrajmaths@yahoo.co.in (Corresponding Author) } }\ \\
		{\footnotesize Department of Mathematics, The Gandhigram Rural Institute  (Deemed to be University),}\\ \footnotesize{Gandhigram -- 624 302, Tamil Nadu, India}\\[3mm]
	}
	\date{}
	\maketitle
	
	
	\begin{abstract}
	In this paper, we introduce three operations on hypergraphs by using tensors. We show that these three formulations are equivalent and we commonly call them as the tensor join. We show that any hypergraph can be viewed as a tensor join of hypergraphs. Tensor join enable us to obtain several existing and new classes of operations on hypergraphs.  We compute the adjacency, the Laplacian, the normalized Laplacian spectrum of weighted hypergraphs constructed by this tensor join. Also we deduce some results on the spectra of hypergraphs in the literature. As an application, we construct several pairs of the adjacency, the Laplacian, the normalized Laplacian cospectral hypergraphs by using the tensor join.
		
		
		\vspace{0.28cm}
		
		\noindent \textbf{Keywords:} Hypergraphs, Tensor join,  Adjacency spectrum, Laplacian spectrum, Normalized Laplacian spectrum, Cospectral hypergraphs.
		\vspace{0.28cm}
		
		\noindent	\textbf{2020 Mathematics Subject Classification:}  05C50, 05C65, 05C76, 15A18
	\end{abstract}
	\section{Introduction}
	In spectral graph theory, the properties of graphs are investigated by the eigenvalues of various associated matrices, such as adjacency matrix, Laplacian matrix, signless Laplacian matrix, normalized Laplacian matrix etc; see,~\cite{cvetkovic2010introduction}. Likewise, in spectral hypergraph theory, spectra of different connectivity tensors and matrices associated to hypergraphs were studied in the literature; see,~\cite{agarwal2006higher,banerjee2021spectrum,banerjee2017spectra,cooper2012spectra,pearson2014spectral,qi2017tensor,rodri2002laplacian}.  Recently,  Anirban Banerjee~\cite{banerjee2021spectrum} introduced some connectivity matrices namely, the adjacency matrix, the Laplacian matrix and the normalized Laplacian matrix for unweighted hypergraphs.
	Therein, some of the properties of hypergraphs were studied using the spectrum of these associated matrices. Subsequently, Amitesh Sarkar and Anirban Banerjee~\cite{sarkar2020joins} extend the definiton of the adjacency matrix of a hypergraph introduced in~\cite{banerjee2021spectrum} to a weighted hypergraph. 
	In the rest of this paper, we consider the matrix representation of hypergraphs defined in~\cite{sarkar2020joins}.
	
	In the literature, several graph operations were defined and the spectra of graphs constructed by these graph operations were determined; see~\cite{cvetkovic2010introduction,gayathri2019adjacency, pavithra2021bowtie, pavithra2021spectra,rajkumar2019spectra,rajkumar2021spectra,rajkumar2020spectra} and the references therein. Recently, Gayathri and Rajkumar~\cite{murugesan2021spectra} introduced a graph operation, namely, $\mathcal{M}$-join. Using this operation several new graph operations were defined and various graph operations in the literature were generalized. There in, the spectral properties of these graphs were investigated. In this direction, there are several hypergraph operations were defined in literature; see the survey paper~\cite{hellmuth2012survey}. In~\cite{sarkar2020joins}, several hypergraph operations, such as the weighted join, the generalized corona were introduced and the adjacency spectra of the hypergraphs formed by these operations were determined. Also some families of cospectral hypergraphs with respect to the adjacency matrix were constructed using these operations. The adjacency spectra of the Cartesian product of hypergraphs was obtained in~\cite{banerjee2021spectrum}.

	Motivated by these, in this paper, we introduce some operations on hypergraphs via tensors. We obtain the spectra of the adjacency, the Laplacian, the normalized Laplacian matrices of the hypergraphs constructed by these operations. 
	
	
	The rest of the paper is arranged as follows: In Section~\ref{Sec1}, we recall some basic notations, definitions and results of graphs/hypergraphs and matrices. In Section~\ref{Sec2}, we introduce a special type of tensor, namely an indicating tensor corresponding to a finite sequence of mutually disjoint sets. Also, we define several particular cases of this tensor. In Section~\ref{defn sec}, we introduce three hypergraph operations by using indicating tensors. We show that these three formulations are equivalent and we commonly call them as the tensor join. We show that any hypergraph can be viewed as a tensor join of hypergraphs.  Tensor join enable us to obtain several existing and new classes of operations on hypergraphs.  In Section~\ref{Sec4}, we compute the spectrum of the adjacency, the Laplacian and the normalized Laplacian matrices of weighted hypergraphs constructed by the tensor join operations introduced in the previous section. Also we deduce some existing results on spectra of hypergraphs. By using the results proved in this section, we construct infinite families of simultaneously adjacency, Laplacian, normalized Laplacian cospectral hypergraphs by using this tensor join operation.
	\section{Preliminaries and notations}\label{Sec1}
	A \textit{hypergraph} $H(V,E)$ consists of a non-empty set $V$ and a multiset $E$ of subsets of $V$.
	The elements of $V$ are called \textit{vertices} and the elements of $E$ are called \textit{hyperedges}, or simply \textit{edges} of $H$. An edge of cardinality one is called a \textit{loop}.
	The \textit{rank} and the \textit{co-rank} of a hypergraph $H$ are defined as $r(H)=\displaystyle \underset{e\in E}{max}\{|e|\}$ and $\rho(H)=\displaystyle \underset{e\in E}{min}\{|e|\}$ respectively. A hypergraph is said to be \textit{uniform} if all of it's edges have the same cardinality. If it is $m$, then the hypergraph is said to be \textit{$m$-uniform}; otherwise, it is called \textit{non-uniform}. A vertex of a hypergraph is said to be isolated if it does not belong to any edge of that hypergraph. 
	Throughout this paper, we consider only hypergraphs having finite number of vertices.
	
	Let $\mathcal{P}^*(A)$ denote the set of all non-empty subsets of a set $A$.
	A hypergraph $H(V,E)$ is said to be \textit{complete}
	if $E=\mathcal{P}^*(V)$. We denote the \textit{complete hypergraph} on $n$ vertices with no loops as $K_n$.
	For, $0\leq r\leq n$, the \textit{complete $r$-uniform hypergraph} on $n$ vertices, denoted by $K_n^r$, is the hypergraph whose edge set is the set of all possible $r$-subsets of $V$.
	
	For a nonempty subset $S$ of positive integers, a
	\textit{$S$-hypergraph} on $V$ is a hypergraph with vertex set $V$ and edge set $E=\bigcup_{s\in S}E_s$,
	where $E_s$ is a non-empty set of $s$-subsets of $V$. The \textit{complement of a
		$S$-hypergraph} $H(V,E)$, denoted by $H^c(V,E^c)$ is the $S$-hypergraph on $V$ whose edge set consists of the
	subsets of $V$ with cardinality in $S$ which do not lie in $E$ \cite{gosselin2012self}. The \textit{degree of a vertex} $v$ in a hypergraph $H$, denoted by $d(v)$, is the number of edges containing $v$ in $H$. 
	
	
	\begin{defn}(\cite{sarkar2020joins}) \normalfont
		Let $H(V,E,W)$ be a             hypergraph with vertex set $V=\{1,2,\dots,n\},$ edge set $E$ and a
		weight function $W:E \rightarrow \mathbb{R}_{\geq 0}$ defined by $W(e) = w_e$ for all $e\in E$. The \textit{adjacency matrix}\label{ad1} $A(H)$ of $H(V,E,W)$ is the $n\times n$ symmetric matrix in which
		$$(i,j)\textnormal{-th entry of } A(H) =\begin{cases}
			\underset{e\in E;i,j\in e}{\sum}\frac{w_e}{|e|-1} & \textnormal{if}~i\neq j, ~i~\textnormal{and}~j~\textnormal{are adjacent;}\\
			~~~~~~~	0 &
			\textnormal{otherwise}.
		\end{cases}$$
	\end{defn}
	If we take $w_e=1$, then $A(H)$ becomes the adjacency matrix of the unweighted hypergraph $H(V,E)$ defined in \cite{banerjee2021spectrum}. 
	The \textit{valency} of a vertex $i$ of $H$, denoted by $d(i)$ is defined as $d(i)=\underset{e\in E;i\in e}{\sum}w_e$. 
	The \textit{Laplacian matrix} $L(H)$ of $H(V,E,W)$ is defined by $L(H)=D(H)-A(H)$, where $D(H)$ is the diagonal matrix whose entries are the valencies $d(i)$ of the vertices $i$ of $H$. If the hypergraph $H(V,E,W)$ has no isolated vertices, then
	its \textit{normalized Laplacian matrix} $\mathcal{L}(H)$ is defined as $\mathcal{L}(H)=D(H)^{-1/2}L(H)D(H)^{-1/2}$.
	
	A weighted/unweighted hypergraph is said to be \textit{$r$-regular} if valency/degree of each of its vertices is $r$.
	
	For a matrix $M$, we use the notation $P_M(x)$ to denote its characteristic polynomial and $\sigma(M)$ to denote its multiset of eigenvalues (spectrum). The spectrum of $A(H)$, $L(H)$ and $\mathcal{L}(H)$ are said to be the \textit{$A$-spectrum}, the \textit{$L$-spectrum} and the \textit{$\mathcal{L}$-spectrum} of the hypergraph $H$, respectively. Two hypergraphs are said to be \textit{$A-$cospectral (resp. $L-$cospectral, $\mathcal{L}-$cospectral)} if they have the same $A$-spectrum (resp. $L$-spectrum, $\mathcal{L}$-spectrum). The largest eigenvalue of $A(H)$ is said to be the \textit{Perron adjacency eigenvalue of $H$}, whereas its other eigenvalues are said to be the \textit{non-Perron adjacency eigenvalues of $H$}.
	
		Let $A_1,A_2,\ldots,A_m$ be square matrices of order $n$ with entries from $\mathbb C$. Then $\lambda_1,\lambda_2,\ldots,\lambda_m$ $\in \mathbb C$ are said to be \textit{
		co-eigenvalues of $A_1,A_2,\ldots,A_m$}, if there exists a vector $X\in\mathbb C^n$ such that $A_iX=\lambda_iX$ for $i=1,2,\ldots,m$~\cite{gayathrithesis}.
	
	Let $I_n$ denote the identity matrix of size $n\times n$ and $J_{n\times m}$ denote the matrix of size $n\times m$ whose all the entries are $1$. In particular, we denote $J_{n\times n}$ simply as $J_n$. The Kronecker product of two matrices $A$ and $B$ is denoted by $A\otimes B$.


	Let $G_1$ and $G_2$ be graphs on $m$ and $n$ vertices, respectively. Let $\pi$ be a binary relation, that is $\pi\subseteq V(G_1)\times V(G_2)$. Then the \textit{$\pi-$graph of $G_1$ and $G_2$}, is the graph whose vertex set is $V(G_1)\cup V(G_2)$ and edge set is $E(G_1)\cup E(G_2)\cup \pi$~\cite{hedetniemi1969classes}.
	An equivalent formulation of this definition is given as follows~\cite{murugesan2021spectra}: Write the binary relation $\pi$ as a $0-1$ matrix $N=(n_{ij})$ of size $m\times n$ in which $n_{ij}=1$ if and only if the $i$-th vertex of $G_1$ and the $j$-th vertex of $G_2$ are related with respect to $\pi$, so the $\pi-$graph of $G_1$ and $G_2$ is the graph obtained by taking one copy of $G_1$ and $G_2$, and joining the $i$-th vertex of $G_1$ to the $j$-th vertex of $G_2$ if and only if $n_{ij}=1$ for $i=1,2,\ldots, n$ and $j=1,2,\ldots,m$. This graph is denoted by $G_1\vee_N G_2$ and is called the~
	\textit{$N$-join of $G$ and $H$}. This definition is extended as follows.
	
	\begin{defn}\label{M join defn} (\cite{murugesan2021spectra})\label{mjoin}
		\normalfont
		Let $\mathcal{H}_k$ be a sequence of $k$ graphs $H_1, H_2,\ldots,H_k$ with $| V(H_i)|$ $=n_i$ for $i=1,2,\ldots,k$ and let $\mathcal M$ $=$ $(M_{12}$, $M_{13},$ $\ldots,$  $M_{1k},$ $M_{23},$ $M_{24},$ $\ldots,$ $M_{2k},$ $\ldots,$ $ M_{(k-1)k})$, where $M_{ij}$ is a $0-1$ matrix of size $n_i\times n_j$. Then \textit{the $\mathcal M$-join of the graphs in $\mathcal H_k$}, denoted by $\bigvee_{\mathcal M}\mathcal H_k$, is the graph $\displaystyle \bigcup_{\substack{i,j=1,\\i<j}}^{k}\left(H_i\vee_{M_{ij}} H_j\right)$.
	\end{defn}
	
	The following results are used in the subsequent sections.
	\begin{thm}(\cite[pp. 483]{meyer2000matrix})\label{sylvester}
		Let $A$ and $B$ be two matrices of sizes $m\times n$ and $n\times m$ respectively. Then for any invertible $m \times m$ matrix $X$,
		$|X+AB|=|X|\times |I_{\mathit {n}}+BX^{-1}A|.$
	\end{thm}
	\begin{thm}	(\cite[Corollary~2]{haynsworth1960reduction}) \label{t1.3}
		Let a real matrix $A$ be partitioned as \begin{center}
			$A=
			\begin{bmatrix}
				A_{11}& A_{{12}} &\cdots&A_{{1k}}\\
				A_{{21}}&A_{22}&\cdots &A_{{2k}}\\
				\vdots&\vdots&\ddots&\vdots\\
				A_{{k1}}&A_{{k2}}&\dots &A_{kk}
			\end{bmatrix}$.
		\end{center} 
		For $i,j=1,2,\ldots,k$, if $A_{ij}$ are symmetric matrices of order $n$ such that they commutes with each other. Then
		$\sigma(A)=\sum_{h=1}^{n}\sigma(E_h),$
		where the summation denotes the union of the multisets and $$E_h=\begin{bmatrix}
			a_{11}^{(h)} & a_{12}^{(h)} &\cdots & a_{1k}^{(h)} \\
			a_{21}^{(h)} & a_{22}^{(h)} & \cdots & a_{2k}^{(h)} \\\vdots & \vdots & \ddots & \vdots
			\\
			a_{k1}^{(h)} & a_{k2}^{(h)} & \cdots &  a_{kk}^{(h)}
		\end{bmatrix},$$
		with $a_{ij}^{(h)}$ is an eigenvalue of $A_{ij}$ corresponding to the same eigenvector $X$ for each $i, j=1,2,\ldots,k$; $h=1,2,\ldots,n$.
	\end{thm}
	\section{Indicating tensors}\label{Sec2}
	Let $\mathcal{R}(a_1,a_2,\dots,a_m)$ denote the range set of the sequence $(a_i)_{i=1}^m$. For $i=1,2,\dots,m$, let
	\begin{center}
		$\mathcal{R}^{a_i}(a_1,a_2,\dots,a_m)=\begin{cases}
			\mathcal{R}(a_1,a_2,\dots,a_m)\backslash \{a_i\} & \text{if}~ a_i\in \{a_1,a_2,\dots,a_m\};\\
			\mathcal{R}(a_1,a_2,\dots,a_m) & \text{otherwise}.
		\end{cases}$ 	
	\end{center} 
	For $n\in \mathbb{N}$, let $[n]:=\{1,2,\dots,n\}$. We denote $\mathcal{P}^*([n])\backslash \underset{y\in [n]}{\cup}\{y\}$ simply by $\widehat{[n]}$.
	\begin{defn}\label{it}
		\normalfont 	For $i=1,2,\dots,k$, let $A_i$ be mutually disjoint sets having $n_i$ elements. Let $\mathcal{A}$ be the sequence $(A_i)_{i=1}^{k}$.
		Then \textit{an indicating tensor corresponding to} $\mathcal{A}$, denoted by $T[\mathcal{A}]:=(T[\mathcal{A}]_{p_1p_2\dots p_N})$, is a $0-1$ tensor of order $N:=n_1+n_2+\dots+n_k$ and dimension $\displaystyle(~\underset{n_1~ times}{\underbrace{n_1+1,\dots,n_1+1}},\underset{n_2~ times}{\underbrace{n_2+1,\dots,n_2+1}},~\dots,~ \underset{n_k~ times}{\underbrace{n_k+1,\dots,n_k+1}}~)$, where $p_1,p_2,\dots, p_{n_1} \in A_1\cup\{\blacktriangledown\}$, $p_{n_1+n_2+\dots+n_i+1},\dots, p_{n_1+n_2+\dots+n_{i+1}}\in A_{i+1}\cup \{\blacktriangledown\}$ for $i=1,2,\ldots, k-1$; $\blacktriangledown$ is an arbitrary symbol that is not an element of any $A_i$, $i=1,2,\dots, k-1$; and is
		satisfying the following:
			\begin{itemize}
			\item[(i)] If there exists $p_1,p_2,\dots, p_N$ such that $\mathcal{R}^{\blacktriangledown}(p_1,p_2,\dots, p_N)\subseteq A_i$ for some $i\in [k]$, then $T[\mathcal{A}]_{p_1p_2\dots p_N} = 0$.
			\item[(ii)]  If there exists $p_1,p_2,\dots, p_N$ such that
			$T[\mathcal{A}]_{p_1p_2\dots p_N} = 1$, then
			$T[\mathcal{A}]_{p_1'p_2'\dots p_N'} = 1$ whenever $\mathcal{R}^{\blacktriangledown}(p_1',p_2',\dots, p_N')=\mathcal{R}^{\blacktriangledown}(p_1,p_2,\dots, p_N)$.
		\end{itemize}
	\end{defn}
	
	Notice that if $p_1=p_2=\dots=p_N=\blacktriangledown$, then we have $\mathcal{R}^{\blacktriangledown}(p_1,p_2,\dots, p_N)=\Phi\subseteq A_i$ and so $T[\mathcal{A}]_{p_1p_2\dots p_N} = 0$.
	\begin{example}\normalfont
		Let $A_1=\{1\}$, $A_2=\{2,3\}$ and $A_3=\{4,5,6\}$. Let $\mathcal{A}=(A_i)_{i=1}^{3}$.  Then an indicating tensor $T[\mathcal{A}]$  of order $6$ and dimension $(2,3,3,4,4,4)$ whose entries are given by,
		$$T[\mathcal{A}]_{i_1i_2\dots i_{6}}=
		\begin{cases}
			1 & \textnormal{if}~\mathcal{R}^{\blacktriangledown}(i_1,i_2,\dots,i_{6})=\{1,2,4,5,6\}~\textnormal{or}~\{1,3\};\\
			0 &
			\textnormal{otherwise}.
		\end{cases}$$
	More explicitly, the entries $T[\mathcal{A}]_{122456}$, $T[\mathcal{A}]_{122465}$, $T[\mathcal{A}]_{122546}$, $T[\mathcal{A}]_{122564}$, $T[\mathcal{A}]_{122645}$, $T[\mathcal{A}]_{122654}$,
	$T[\mathcal{A}]_{12\blacktriangledown456}$, $T[\mathcal{A}]_{12\blacktriangledown465}$, $T[\mathcal{A}]_{12\blacktriangledown546}$, $T[\mathcal{A}]_{12\blacktriangledown564}$, $T[\mathcal{A}]_{12\blacktriangledown645}$, $T[\mathcal{A}]_{12\blacktriangledown654}$,
	$T[\mathcal{A}]_{1\blacktriangledown2456}$, $T[\mathcal{A}]_{1\blacktriangledown2465}$, $T[\mathcal{A}]_{1\blacktriangledown2546}$, $T[\mathcal{A}]_{1\blacktriangledown2564}$, $T[\mathcal{A}]_{1\blacktriangledown2645}$, $T[\mathcal{A}]_{1\blacktriangledown2654}$,
	$T[\mathcal{A}]_{133\blacktriangledown\blacktriangledown\blacktriangledown}$, $T[\mathcal{A}]_{1\blacktriangledown3\blacktriangledown\blacktriangledown\blacktriangledown}$, $T[\mathcal{A}]_{13\blacktriangledown\blacktriangledown\blacktriangledown\blacktriangledown}$ take the value $1$ and the remaining entries are zero.
	\end{example}
	
	\begin{defn}
		\normalfont We  call an indicating tensor obtained by taking $A_i$ instead of  $A_i\cup\{\blacktriangledown\}$ for $i=1,2,\dots,k$ in Definition~\ref{it} as \textit{an indicating tensor of type-2 corresponding to $\mathcal{A}$} and is denoted by $T^*[\mathcal{A}]$.
	\end{defn}
	\begin{example}\normalfont
	Let $A_1=\{1\}$, $A_2=\{2,3\}$ and $A_3=\{4\}$. Let $\mathcal{A}=(A_i)_{i=1}^{3}$.  Then an indicating tensor $T^*[\mathcal{A}]$ of type-2 of order $4$ and dimension $(1,2,2,1)$ whose entries are given by,
	$$T^*[\mathcal{A}]_{i_1i_2i_3i_4}=
	\begin{cases}
		1 & \textnormal{if}~\mathcal{R}^{\blacktriangledown}(i_1,i_2,i_3,i_4)=\{1,2,4\}~\textnormal{or}~\{1,3,4\}~\textnormal{or}~\{1,2,3,4\};\\
		0 &
		\textnormal{otherwise}.
	\end{cases}$$
	More explicitly, $T^*[\mathcal{A}]_{1224}= T^*[\mathcal{A}]_{1334}=T^*[\mathcal{A}]_{1234}=T^*[\mathcal{A}]_{1324}=1$ and the remaining entries are zero.
\end{example}
	 For an indicating tensor $T[\mathcal{A}]$ and an indicating tensor $T^*[\mathcal{A}]$ of type-2, we define the following notations.
	\begin{enumerate}
		
		\item[(i)]	
		$E(T[\mathcal{A}]):=\{\mathcal{R}^{\blacktriangledown}(p_1,p_2,\dots, p_N)~|~T[\mathcal{A}]_{p_1p_2\dots p_N}=1\}$.

		\item[(ii)]	 $E(T^*[\mathcal{A}]):=\{\mathcal{R}(p_1,p_2,\dots, p_N)~|~T^*[\mathcal{A}]_{p_1p_2\dots p_N}=1\}$.
		
		\item[(iii)]  For each $p\in A_i,~ q\in A_j$ $(1\leq i\leq j\leq k)$, $c\in [N]$, 
		
		 $E_{p,q}^{c}(T[\mathcal{A}]):=\{S\in E(T[\mathcal{A}])~|~ \{p,q\}\subseteq S,~|S|=c\}$.
		
	\end{enumerate}
	In the following we introduce some special classes of indicating tensors. 
	\begin{enumerate}
		
		\item[(1)] For each $m\in\{1,2,\dots,N\}$, let $T[\mathcal{A};m]$ denote an indicating tensor corresponding to $\mathcal{A}$ in which $~T[\mathcal{A};m]_{p_1p_2\dots p_N}=0$ whenever $|\mathcal{R}^{\blacktriangledown}(p_1,p_2,\dots, p_N)|\neq m$.
		
		%
		%
		
		\item[(2)] For a non empty subset $B$ of $\{k,k+1,\dots,N\}$, let ${}_BT[\mathcal{A}]$ denote the indicating tensor corresponding to $\mathcal{A}$ in which
		\[
		{}_BT[\mathcal{A}]_{p_1p_2\dots p_N}=	
		\begin{cases}
			1 &\vspace{-0.28cm}\text{if}~ |\mathcal{R}^{\blacktriangledown}(p_1,p_2,\dots, p_N)|\in B~\text{and}~\\&
			\mathcal{R}^{\blacktriangledown}(p_1,p_2,\dots, p_N)\cap A_i\neq \Phi~$ for all $i\in[k];\\
			~0 &\text{otherwise}.
		\end{cases}\]	
		
		
		
		\item[(3)]
		Let $J[\mathcal{A}]$ denote the indicating tensor corresponding to $\mathcal{A}$ in which
		\begin{center}
			$J[\mathcal{A}]_{p_1p_2\dots p_N} $ =			
			$ \begin{cases}
				0 &\text{if}~     
				\mathcal{R}^{\blacktriangledown}(p_1,p_2,\dots, p_N)\subseteq A_i~\text{for some}~ i\in[k];  \\
				1 &\text{otherwise}.
			\end{cases} $	
		\end{center} 
		\item[(4)]	For $i=1,2,\dots,k$, let $A_i=\{u_{i_1},u_{i_2},\dots, u_{i_{n}}\}$. For each $r\in[n]$, let $_rT[\mathcal{A}]$ denote the indicating tensor corresponding to $\mathcal{A}$ with
		\[
		_rT[\mathcal{A}]_{p_1p_2\dots p_{nk}} =			
		\begin{cases}
			1 &\vspace{-0.25cm}\text{if}~ \mathcal{R}^{\blacktriangledown}(p_1,p_2,\dots, p_{nk})= \underset{i=1}{\overset{k}{\bigcup}}\{u_{i_{l_1}},u_{i_{l_2}},\dots,u_{i_{l_r}}\}\\
			&\text{for some}~\{l_1,l_2,\dots, l_r\}\subseteq[n];\\
			0 &\text{otherwise}.
		\end{cases}\]  
		
		\item[(5)]			
		Let $I[\mathcal{A}]:= ~_1T[\mathcal{A}]$ and we call this as the \textit{identity indicating tensor corresponding to~$\mathcal{A}$}.
		
		\item[(6)] Let $H(V(H),E(H))$ be a hypergraph with $V(H)=\{1, 2,\dots, n\}$. Let $1<k\leq \rho(H)$ and let $(G_i(U_i, E_i))_{i=1}^k$ be a sequence of hypergraphs with $U_i=\{u_{i1},u_{i2},\dots,u_{in}\}$. Let $\mathcal{A}=(U_i)_{i=1}^k$. 
		Let $N_H[\mathcal{A}]$ denote the indicating tensor corresponding to $\mathcal{A}$ with
		\[N_H[\mathcal{A}]_{p_1p_2\dots p_{nk}} =			
		\begin{cases}
			1 \vspace{-0.25cm}&\text{if}~ \mathcal{R}^{\blacktriangledown}(p_1,p_2,\dots, p_{nk})= \underset{i=1}{\overset{k}{\bigcup}}\{u_{i_{l_1}},u_{i_{l_2}},\dots,u_{i_{l_{s_i}}}\}\\
			\vspace{-0.25cm}&\text{where}~s_i\geq1,~D_i=\{l_1,l_2,\dots, l_{s_i}\}\subseteq V(H)~\text{such that}\\&\text{the set of all}~ D_i~\text{forms a partition of}~e~\text{for some}~e\in E(H).\\
			0 \quad &\text{otherwise}.
		\end{cases}\]
		
		\item[(7)] For $i=1,2,\dots,k$, let $|A_i|=n$. We denote the indicating tensor $J[\mathcal{A}]-{_rT[\mathcal{A}]}$ by $\Im_r[\mathcal{A}]$. When $r=1$, we denote it simply by $\Im[\mathcal{A}]$.

		\item[(8)]	We denote the indicating tensor $N_H[\mathcal{A}]+{_rT[\mathcal{A}]}$ by $_{H_r}N[\mathcal{A}]$.  When $r=1$, we denote it simply by $_{H}N[\mathcal{A}]$.

	\end{enumerate} 
	
	\section{Tensor join of hypergraphs}\label{defn sec}
	In the rest of the paper, whenever we consider a sequence of weighted/unweighted hypergraphs $(G_i)_{i=1}^{k}$, without loss of generality, we assume that the vertex sets of $G_i$s are mutually disjoint for $i=1,2,\dots,k.$
	
	\begin{defn}\label{defn1}
		\normalfont
		Let $\mathcal{G}=\left(G_i(V_i,E_i)\right)_{i=1}^k$ be a sequence of $k$ hypergraphs. Let $\mathcal{V}=(V_i)_{i=1}^k$. Consider an indicating tensor $T[\mathcal{V}]$. Then the \textit{$T[\mathcal{V}]$-join of hypergraphs in $\mathcal{G}$}, denoted by $\underset{T[\mathcal{V}]}{\bigvee}\mathcal{G}$, is the hypergraph constructed as follows:
		\begin{itemize}
			\item Take one copy of $G_i,~i=1,2,\dots, k$;
			\item For each $D\subseteq  \underset{i=1}{\overset{k}{\bigcup}}V_i$, join the vertices in $D$ as an edge in $\underset{T[\mathcal{V}]}{\bigvee}\mathcal{G}$ if and only if $D\in E(T[\mathcal{V}])$.
		\end{itemize}
		If $\mathcal{G}=(G_1,G_2)$, then we denote the $T[\mathcal{V}]$-join of hypergraphs in $\mathcal{G}$ by $G_1 \underset{T[\mathcal{V}]}{\bigvee}G_2$.
	\end{defn}
	\begin{example}
		\normalfont
		Consider the hypergraphs $G_1(V_1,E_1)$, $G_2(V_2,E_2)$ and $G_3(V_3,E_3)$ as shown in Figures~\ref{fig2}(a),~\ref{fig2}(b),~\ref{fig2}(c) respectively. Let $\mathcal{G}=(G_i)_{i=1}^3$ and $\mathcal{V}=(V_i)_{i=1}^3$. 
		\begin{figure}[ht]
			\begin{center}
				\includegraphics[scale=1.3]{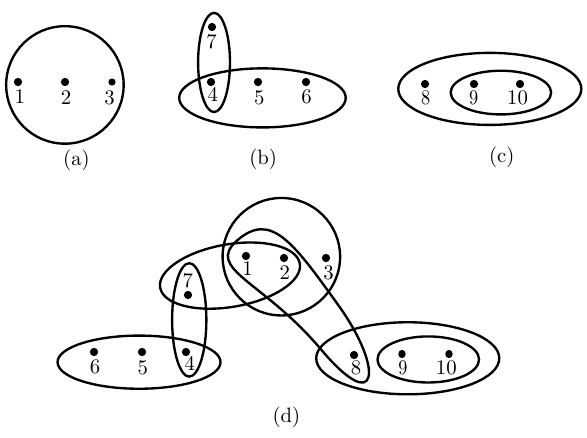}
			\end{center}\caption{The hypergraphs (a) $G_1(V_1,E_1)$, (b) $G_2(V_2,E_2)$, (c)  $G_3(V_3,E_3)$ and (d) $\underset{T[\mathcal{V}]}{\bigvee}\mathcal{G}$}\label{fig2}
		\end{figure}		
		Consider the indicating tensor $T[\mathcal{V}]$ of order $10$ and dimension $(4,4,4,5,5,5,5,4,4,4)$ with
		$$T[\mathcal{V}]_{i_1i_2\dots i_{10}}=
		\begin{cases}
			1 & \textnormal{if}~\mathcal{R}^{\blacktriangledown}(i_1,i_2,\dots,i_{10})=\{1,2,7\}~\textnormal{or}~\{1,2,8\};\\
			0 &
			\textnormal{otherwise}.
		\end{cases}$$
		Notice that, $E(T[\mathcal{V}])=\{\{1,2,7\}, \{1,2,8\}\}$. Then the hypergraph $\underset{T[\mathcal{V}]}{\bigvee}\mathcal{G}$ is as shown in Figure~\ref{fig2}(d).
	\end{example}
	\begin{defn}\label{defn2}
		\normalfont
		Let $\mathcal{G}=(G_i(V_i,E_i))_{i=1}^k$ be a sequence of $k$ hypergraphs. For each $S\in \widehat{[k]}$, let $\mathcal{V}_S=(V_i)_{i\in S}$. Let $\mathcal{T}^*=\{T^*[\mathcal{V}_S]~|~S\in \widehat{[k]}\}$ be a set of indicating tensors of type-2. Then the \textit{$\mathcal{T}^*$-join of hypergraphs in $\mathcal{G}$}, denoted by $\underset{\mathcal{T}^*}{\bigvee}\mathcal{G}$, is the hypergraph obtained by taking a copy of each $G_i$ and for each $D\subseteq \underset{i=1}{\overset{k}{\bigcup}}V_i$, join the set of vertices in $D$ by an edge in $\underset{\mathcal{T}^*}{\bigvee}\mathcal{G}$ if and only if $D\in E(T^*[\mathcal{V}_S])$ for some $S\in\widehat{[k]}$.
	\end{defn}
	\begin{defn}\label{H1}
		\normalfont
		Let $H$ be a hypergraph with $V(H)=[k]$. Let $\mathcal{G}=(G_i(V_i,E_i))_{i=1}^{k}$ be a sequence of hypergraphs with $|V_i|=n_i$ for $i=1,2,\dots,k.$ For each $e\in E(H)$, let $\mathcal{V}_e=(V_i)_{i\in e}$, $N_e:={\underset{i\in e}{\sum}n_i}$ and $\mathcal{G}_e=\{G_i~|~i\in e\}$. Let $\mathcal{T}=\{ T[\mathcal{V}_e]~|~e\in E(H)\}$, where for each $e\in E(H)$, $T[\mathcal{V}_e]$ is a non-zero indicating tensor with
		\begin{align*}\label{H0}
			T[\mathcal{V}_e]_{p_1p_2\dots p_{{N_e}}}=
			0 ~\text{if}~ ~\mathcal{R}^{\blacktriangledown}(p_1,p_2,\dots,p_{{N_e}})\cap V_i=\Phi~ \text{for some} ~i\in e.
		\end{align*}
		Then construct the hypergraph
		by taking a copy of each $G_i$ and doing the $T[\mathcal{V}_e]$-join of hypergraphs in $\mathcal{G}_e$ for each edge $e\in E(H)$. We denote this hypergraph by $\mathcal{G}(H,\mathcal{T})$ and call it as the \textit{$(H,\mathcal{T})$-join of hypergraphs in $\mathcal{G}$}.
	\end{defn} 
	Notice that $V(\mathcal{G}(H,\mathcal{T}))=\underset{i=1}{\overset{k}{\bigcup}}V_i$ and $E(\mathcal{G}(H,\mathcal{T}))=\underset{i=1}{\overset{k}{\bigcup}}E(G_i)\underset{e\in E(H)}{\bigcup}E(T[\mathcal{V}_e])$.
	\begin{thm}\label{r3.1}
		 Definitions~$\ref{defn1}$, $\ref{defn2}$ and $\ref{H1}$ are equivalent.
		
	\end{thm}
	\begin{proof}
		Let $\mathcal{G}=(G_i(V_i,E_i))_{i=1}^k$ be a sequence of $k$ hypergraphs with $|V_i|=n_i$ for $i=1,2,\dots,k$.
		\begin{itemize}
			\item[(1)] Consider an indicating tensor $T[\mathcal{V}]$, where $\mathcal{V}=(V_i)_{i=1}^k$ and assume that we have constructed the hypergraph $\underset{T[\mathcal{V}]}{\bigvee}\mathcal{G}$ as per Definition~\ref{defn1}. We show that this hypergraph can be viewed as the hypergraph $\underset{\mathcal{T}^*}{\bigvee}\mathcal{G}$ for some suitable $\mathcal{T}^*$ as per Definition~\ref{defn2}. For each $S\in\widehat{[k]}$, let $\mathcal{V}_S=(V_i)_{i\in S}$ and $w(S)=\underset{r\in S}{\sum}n_r$. Take $\mathcal{T}^*=\{T^*[\mathcal{V}_S]~|~S\in \widehat{[k]}\}$, where $T^*[\mathcal{V}_S]$ is the indicating tensor of type-2 with \[
			T^*[\mathcal{V}_S]_{p_1p_2\dots p_{w(S)}}=T[\mathcal{V}]_{q_1q_2\dots q_N},\] where $q_1, q_2,\dots,q_N$ are such that $\mathcal{R}^{\blacktriangledown}(q_1,q_2,\dots, q_N)=\mathcal{R}(p_1,p_2,\dots, p_{w(S)})$. Now construct the hypergraph $\underset{\mathcal{T}^*}{\bigvee}\mathcal{G}$ as per Definition~\ref{defn1}. Then this hypergraph is the same as the hypergraph $\underset{T[\mathcal{V}]}{\bigvee}\mathcal{G}$.
			\item[(2)]  Let $\mathcal{T}^*=\{T^*[\mathcal{V}_S]~|~S\in \widehat{[k]}\}$ be a set of indicating tensors of type-2, where $\mathcal{V}_S=(V_i)_{i\in S}$ for all $S\in \widehat{[k]}$. Assume that we have constructed the hypergraph $\underset{\mathcal{T}^*}{\bigvee}\mathcal{G}$ as per Definition~\ref{defn2}. We show that this hypergraph is the same as the hypergraph $\mathcal{G}(H,\mathcal{T})$ for some suitable hypergraph $H$ and a set of indicating tensors $\mathcal{T}$ as per Definition~\ref{H1}. First construct the hypergraph $H$ by using $\mathcal{T}^*$ as follows: Take $V(H)=[k]$. For each $T^*[\mathcal{V}_S]\in \mathcal{T}^*$, make $S\subseteq V(H)$ as an edge in $H$ if and only if $T^*[\mathcal{V}_S]$ is non-zero. Now, for each $e\in E(H)$, let $N_e=\underset{r\in e}{\sum}n_r$. 
			Take $\mathcal{T}=\{T[\mathcal{V}_e]~|~e\in E(H)\}$, where $T[\mathcal{V}_e]$ is the indicating tensor with
			\[T[\mathcal{V}_e]_{p_1p_2\dots p_{N_{e}}}=T^*[\mathcal{V}_e]_{q_1q_2\dots q_{N_e}},\] where $q_1,q_2,\dots, q_{N_e}$ are such that $\mathcal{R}(q_1,q_2,\dots, q_{N_e})=\mathcal{R}^{\blacktriangledown}(p_1,p_2,\dots, p_{N_{e}}).$
			Now, construct the hypergraph $\mathcal{G}(H, \mathcal{T})$ as per Definition~\ref{H1}. Then this hypergraph is the same as the hypergraph $\underset{\mathcal{T}^*}{\bigvee}\mathcal{G}$.
			
			\item[(3)] 	Let $H$ be a hypergraph with $V(H)=[k]$.  For each $e\in E(H)$, let $\mathcal{V}_e=(V_i)_{i\in e}$. Let $\mathcal{T}=\{T[\mathcal{V}_e]~|~e\in E(H)\}$. Assume that we have constructed $\mathcal{G}(H,\mathcal{T})$ as per Definition~\ref{H1}. We show that this hypergraph can be viewed as $\underset{T[\mathcal{V}]}{\bigvee}\mathcal{G}$ for some suitable indicating tensor $T[\mathcal{V}]$, where $\mathcal{V}=(V_i)_{i=1}^k$. Take the indicating tensor $T[\mathcal{V}]$ with
			\[				T[\mathcal{V}]_{p_1p_2\dots p_{N}} =\begin{cases}
				T[\mathcal{V}_e]_{q_1q_2\dots q_{N_e}}&\vspace{-0.3cm} \text{if there exists}~ e\in E(H)~ \text{such that}\\&\vspace{-0.3cm}\mathcal{R}^{\blacktriangledown}(p_1,p_2,\dots, p_N)\cap V_i\neq\Phi~\text{for all}~ i\in e~\\
				&\text{with}~\mathcal{R}^{\blacktriangledown}(p_1,p_2,\dots, p_N)=~\mathcal{R}^{\blacktriangledown}(q_1,q_2,\dots, q_{N_e}); \\
				0&\text{otherwise}.
			\end{cases}\]
			Construct the hypergraph $\underset{T[\mathcal{V}]}{\bigvee}\mathcal{G}$ as per Definition~\ref{defn1}, which becomes the same as the hypergraph $\mathcal{G}(H,\mathcal{T})$.
		\end{itemize}
	\end{proof}
In view of Theorem~\ref{r3.1}, hereafter we say `the tensor join of hypergraphs' to mean the hypergraph obtained by any one of the operations defined in Definitions~$\ref{defn1}$,~$\ref{defn2}$ and $\ref{H1}$, unless we specifically mentioned otherwise.
	\begin{note}
		\normalfont
		Any hypergraph can be viewed as a  tensor join of some hypergraphs.
		For, let $H$ be a hypergraph with $|V(H)|=n$. Take a partition $V_i,~i=1,2,\dots,k$ of  $V(H)$, where $k~\leq n$. For each $i=1,2,\dots,k$, let $G_i$ be the subhypergraph of $H$ induced by the vertex subset $V_i$. Let $\mathcal{G}=(G_i)_{i=1}^k$ and $\mathcal{V}=(V_i)_{i=1}^k$. Now consider the indicating tensor $T[\mathcal{V}]$ with
		\begin{center}
			$T[\mathcal{V}]_{p_1 p_2\dots p_{n}}$ =
			$\begin{cases}
				1 & \text{if}~ \mathcal{R}^{\blacktriangledown}(p_1,p_2,\dots, p_{n})\in E(H);\\
				0 & \text{otherwise}.
			\end{cases}$	
		\end{center}
		Then it is clear that $H$ is the same as the hypergraph 	$\underset{T[\mathcal{V}]}{\bigvee}\mathcal{G}$.
	\end{note}
	In the following theorem, we assert that for a given sequence $\mathcal{M}$ of matrices, the $\mathcal M$-join of graphs in a sequence $\mathcal{G}$ defined in Definition~\ref{M join defn} can be viewed as a $T[\mathcal{A}]$-join of graphs in $\mathcal{G}$ for some suitable $T[\mathcal{A}]$ and vice versa.
	\begin{thm}\label{rem M to T}\normalfont
		Let $\mathcal{G}=(G_i)_{i=1}^k$ be a sequence of graphs with $V(G_i)=\{u_{i1},u_{i2},\dots,u_{in_i}\}$ for $i=1,2,\dots,k$ and let $\mathcal{V}=(V_i)_{i=1}^k$. Then corresponding to a given sequence $\mathcal{M}=(M_{12}, M_{13}, \dots, M_{1k},$ $M_{23}, M_{24}, \dots, M_{2k}, \dots, M_{(k-1)k})$, where $M_{ij}$ is a $0-1$ matrix of size $n_i \times n_j$, there exist an indicating tensor $T[\mathcal{V}]$ such that the graph $\bigvee_{\mathcal{M}}\mathcal{G}$  is the same as the graph $\underset{T[\mathcal{V}]}{\bigvee}\mathcal{G}$ and vice versa.
	\end{thm}
	\begin{proof} 
		Assume that the graph $\bigvee_{\mathcal{M}}\mathcal{G}$ is constructed as per Definition~\ref{mjoin}. Let us denote the $(r,t)$-th entry of $M_{ij}$ by $(M_{ij})_{rt}$. Now consider the indicating tensor $T[\mathcal{V}]$ with \begin{center}
			$T[\mathcal{V}]_{p_1p_2\dots p_N}$ =			
			$ \begin{cases}
				(M_{ij})_{rt} &\text{if}~    
				\mathcal{R}^{\blacktriangledown}(p_1,p_2,\dots, p_N)=\{u_{ir}, u_{jt}\}; \\
				0 &\text{otherwise}.
			\end{cases} $	
		\end{center} Then the graph $\underset{T[\mathcal{V}]}{\bigvee}\mathcal{G}$ constructed as per Definition~\ref{defn1} is the same as the graph $\bigvee_{\mathcal{M}}\mathcal{G}$.
		
		Conversely, assume that an indicating tensor $T[\mathcal{V}]$  corresponding to $\mathcal{V}$ is given and the hypergraph $\underset{T[\mathcal{V}]}{\bigvee}\mathcal{G}$ is constructed as per Definition~\ref{defn1}. For, $1\leq i\leq j\leq k$, consider the matrix $M_{ij}$ whose $(r,t)$-th entry is defined as $(M_{ij})_{rt}=T[\mathcal{V}]_{p_1 p_2\dots p_{N}}$, where  $p_{n_1+n_2+\dots+n_{i-1}+1}=\dots= p_{n_1+n_2+\dots+n_{i}}=u_{ir}$, $p_{n_1+n_2+\dots+n_{j-1}+1}=\dots= p_{n_1+n_2+\dots+n_{j}}=u_{jt}$ and all other indices are zero. Then the graph $\bigvee_{\mathcal{M}}\mathcal{G}$ constructed as per Definition~\ref{mjoin} is the same as the graph $\underset{T[\mathcal{V}]}{\bigvee}\mathcal{G}$.		
	\end{proof}	
	Naturally, there are several ways of constructing the matrix $M_{ij}$ from the given indicating tensor $T[\mathcal{V}]$. In Theorem~\ref{rem M to T}, we exhibit a way of constructing such matrices. Also, notice that the indicating tensor $T[\mathcal{V}]$ referred in Theorem~\ref{rem M to T} is especially the indicating tensor $T[\mathcal{V};2]$.
	\subsection{Some classes of hypergraphs as $T[\mathcal{A}]$-join of hypergraphs}
	
	
	In Table~\ref{tab3}, we list some existing and new classes of hypergraphs which can be expressed as a ${}_BT[\mathcal{A}]$-join of hypergrphs in $\mathcal{G}=(H_i)_{i=1}^k$, by suitably taking the hypergraphs $H_i$s, the set $B$ and the value $k$ as shown in the same table correspond to each class of hypergaphs, where $\mathcal{A}=(V(H_i))_{i=1}^k$.
	
	\setlength\extrarowheight{1pt}
	\renewcommand{\arraystretch}{1.0}
	\tabcolsep=0.11cm
	\begin{longtable}{ | M{1.2cm}| M{6.5cm}|M{1cm}|M{1.5 cm}|M{3.5cm}|}
		\hline
		\textbf{S. No.}  &  \textbf{Name of the hypergraph}  &
		$H_i$& $k$ & $B$ \\
		\hline
		

		1. & Complete $m$-uniform $m$-partite hypergraph \cite{voloshin2009introduction} & $K_{n_i}^c$ & $m$ & $\{m\}$   \\
		\hline

		2. &Complete $m$-uniform weak $k$-partite hypergraph, $k\leq m$ \cite{sarkar2020joins}& $K_{n_i}^c$ & $k$ & $\{m\}$  \\
		\hline
		
		
		3. & Complete weak $k$-partite hypergraph &$K_{n_i}^c$ & $k$ & $\{k,k+1,\dots,N\}$ \\
		\hline
		
		
		4. &  Join of a set $\mathcal{G}$ of non-uniform hypergraphs \cite{sarkar2020joins} &$H_i$ & $k$ & a subset of $\{k,k+1,\dots,N\}$ \\
		\hline 
		
		5. &  Join of a set $\mathcal{G}$ of $m$-uniform hypergraphs \cite{sarkar2020joins} &$H_i$ & $k(\leq m)$& $\{m\}$  \\
		\hline 
		
		\caption{Viewing some existing and new class of hypergraphs as a $T[\mathcal{A}]$-join of hypergraphs in $\mathcal{G}$}\label{tab3}
	\end{longtable}
	Notice that if $m\geq 2$, the complete $m$-uniform weak $2$-partite hypergraph becomes the complete $m$-uniform bipartite hypergraph. Also the complete weak $2$-partite hypergraph becomes the complete bipartite hypergraph.
	
	
	
	\subsection{Some unary hypergraph operations as $T[\mathcal{A}]$-join of hypergraphs}
	First we define a new type of complement of a hypergraph.
	\begin{defn}
		\normalfont 
		Let $H(V,E)$ be a hypergraph. We define the \textit{total complement of $H$}, denoted by $\overline{H}(V, \overline{E})$, as the hypergraph with vertex set $V$ and the edge set $\overline{E}=\mathcal{P}^*(V)\backslash (E\cup S)$, where $S$ is the set of all singletons of $V$. \end{defn}
	In Table~\ref{utab}, we define several new unary hypergraph operations and name them analogous to the unary operations on graphs defined in Section~4.1 of~\cite{murugesan2021spectra}. For the operations given in S.Nos. $37$-$126$ of this table, we assume that $H$ contains no loops.
	\setlength\extrarowheight{1pt}
	\renewcommand{\arraystretch}{1.0}
	\tabcolsep=0.11cm
	\begin{longtable}{ | M{1.2cm} | M{2.5cm} |m{11.5cm}| }
		\hline
		\textbf{S. No.}  & \textbf{Description} &
		~~~~~~~~~~~~~~~~~~~~~~~~~~~~\textbf{Name of the hypergraph}\\ 
		\hline
		
		1. & $H \underset{_rT[\mathcal{V}]}{\bigvee}H$ & ~$r$-Mirror hypergraph of $H$  \\
		\hline
		
		2. & $H \underset{_rT[\mathcal{V}]}{\bigvee}H^c$ & ~$r$-Mirror complemented neighbourhood hypergraph of $H$  \\
		\hline
		
		3. & $H \underset{_rT[\mathcal{V}]}{\bigvee}K_n$ &  ~$C$-$r$-complete hypergraph of $H$  \\
		\hline	
		
		4. & $H \underset{_rT[\mathcal{V}]}{\bigvee}K_n^{c}$ & ~$C$-$r$-hypergraph of $H$ \\
		\hline
		
		5. & $H \underset{_rT[\mathcal{V}]}{\bigvee}\overline{H}$ &  ~$r$-Mirror total complemented neighbourhood hypergraph of $H$  \\
		\hline
		
		
		%
		%
		%
		
		6. & $H \underset{J[\mathcal{V}]}{\bigvee}H$ &  ~Join neighbourhood hypergraph of $H$ \\
		\hline
		
		7. & $H \underset{J[\mathcal{V}]}{\bigvee}H^c$ &  ~Join complemented neighbourhood hypergraph of $H$   \\
		\hline
		
		8. & $H \underset{J[\mathcal{V}]}{\bigvee}K_n$ & ~Join complete hypergraph of $H$ \\
		\hline
		
		9. & $H \underset{J[\mathcal{V}]}{\bigvee}K_n^{c}$ & ~Join hypergraph of $H$ \\
		\hline
		
		10. & $H \underset{J[\mathcal{V}]}{\bigvee}\overline{H}$ &  ~Join total complemented neighbourhood hypergraph of $H$\\
		\hline
		
		11. & $H \underset{\Im_r[\mathcal{V}]}{\bigvee}H$ & ~$VC$-$r$-neighbourhood hypergraph of $H$  \\
		\hline

		12. & $H \underset{\Im_r[\mathcal{V}]}{\bigvee}H^c$ & ~$VC$-$r$-complemented neighbourhood hypergraph of $H$ \\
		\hline
		
		13. & $H \underset{\Im_r[\mathcal{V}]}{\bigvee}K_n$ & ~$VC$-$r$-complete hypergraph of $H$  \\
		\hline
		
		14. & $H \underset{\Im_r[\mathcal{V}]}{\bigvee}K_n^{c}$ & ~$VC$-$r$-hypergraph of $H$  \\
		\hline
		
		15. & $H \underset{\Im_r[\mathcal{V}]}{\bigvee}\overline{H}$ &  ~$VC$-$r$-total complemented neighbourhood hypergraph of $H$   \\
		\hline
		
		%
		%
		%
		%
		
		16. & $H^c \underset{_rT[\mathcal{V}]}{\bigvee}H^c$ & ~$r$-Mirror-complement hypergraph of $H$  \\
		\hline
		
		17. & $H^c \underset{_rT[\mathcal{V}]}{\bigvee}K_n$ &  ~$C$-$r$-complete complement hypergraph of $H$ \\
		\hline	
		
		18. & $H^c \underset{_rT[\mathcal{V}]}{\bigvee}K_n^{c}$ &  ~$C$-$r$-complement hypergraph of $H$   \\
		\hline
		
		%
		%
		
		19. & $H^c \underset{J[\mathcal{V}]}{\bigvee}H^c$ & ~Join neighbourhood-complement hypergrph of $H$  \\
		\hline
		
		20. & $H^c \underset{J[\mathcal{V}]}{\bigvee}K_n$ &  ~Join complete-complement hypergrph of $H$ \\
		\hline
		
		21. & $H^c \underset{J[\mathcal{V}]}{\bigvee}K_n^{c}$ & ~Join-complement hypergrph of $H$  \\
		\hline
		
		22. & $H^c \underset{\Im_r[\mathcal{V}]}{\bigvee}H^c$ & ~$VC$-$r$-neighbourhood-complement hypergraph of $H$   \\
		\hline
		
		23. & $H^c \underset{\Im_r[\mathcal{V}]}{\bigvee}K_n$ & ~$VC$-$r$-complete-complement hypergraph of $H$  \\
		\hline
		
		24. & $H^c \underset{\Im_r[\mathcal{V}]}{\bigvee}K_n^{c}$ & ~$VC$-$r$-complement hypergraph of $H$   \\
		\hline
		
		%
		%
		
		25. & $\overline{H} \underset{_rT[\mathcal{V}]}{\bigvee}H^c$ & ~Total $r$-mirror complement hypergraph of $H$ \\
		\hline
		
		26. & $\overline{H} \underset{_rT[\mathcal{V}]}{\bigvee}K_n$ & ~$C$-$r$-complete total complement hypergraph of $H$ \\
		\hline	
		
		27. & $\overline{H} \underset{_rT[\mathcal{V}]}{\bigvee}K_n^{c}$ & ~$C$-$r$-total complement hypergraph of $H$ \\
		\hline
		
		28. & $\overline{H} \underset{_rT[\mathcal{V}]}{\bigvee}\overline{H}$ & ~$r$-Mirror total complemented hypergraph of $H$  \\
		\hline
		
		%
		%
		%
		
		29. & $\overline{H} \underset{J[\mathcal{V}]}{\bigvee}H^c$ & ~Total join neighbourhood complement hypergraph of $H$ \\
		\hline
		
		30. & $\overline{H} \underset{J[\mathcal{V}]}{\bigvee}K_n$ & ~Join complete total complement hypergraph of $H$ \\
		\hline
		
		31. & $\overline{H} \underset{J[\mathcal{V}]}{\bigvee}K_n^{c}$ & ~Join total complement hypergrph of $H$\\
		\hline
		
		32. & $\overline{H} \underset{J[\mathcal{V}]}{\bigvee}\overline{H}$ & ~Join neighbourhood-total complement hypergraph of $H$ \\
		\hline
		
		33. & $\overline{H} \underset{\Im_r[\mathcal{V}]}{\bigvee}H^c$ &  ~Total $VC$-$r$-neighbourhood complement hypergraph of $H$ \\
		\hline
		
		34. & $\overline{H} \underset{\Im_r[\mathcal{V}]}{\bigvee}K_n$ &~$VC$-$r$-complete total complement hypergraph of $H$ \\
		\hline
		
		35. & $\overline{H} \underset{\Im_r[\mathcal{V}]}{\bigvee}K_n^{c}$ & ~$VC$-$r$-total complement hypergraph of $H$  \\
		\hline
		
		36. & $\overline{H} \underset{\Im_r[\mathcal{V}]}{\bigvee}\overline{H}$ & ~$VC$-$r$-neighbourhood total complement hypergraph of $H$  \\
		\hline
		
		37. & $H \underset{N_H[\mathcal{V}]}{\bigvee}H$ &  ~$N$-neighbourhood hypergraph of $H$  \\
		\hline
		
		38. & $H \underset{N_H[\mathcal{V}]}{\bigvee}H^c$ & ~$N$-complemented neighbourhood hypergraph of $H$ \\
		\hline
		
		39. & $H \underset{N_H[\mathcal{V}]}{\bigvee}K_n$ &  ~$N$-complete hypergraph of $H$ \\
		\hline	
		
		40. & $H \underset{N_H[\mathcal{V}]}{\bigvee}K_n^{c}$ & ~$N$-hypergraph of $H$ \\
		\hline
		
		41. & $H \underset{N_H[\mathcal{V}]}{\bigvee}\overline{H}$ &  ~$N$-total complemented neighbourhood hypergraph of $H$ \\
		\hline
		
		42. & $H \underset{	_{H_r}N[\mathcal{V}]}{\bigvee}H$ &  ~$\overline{N}$-$r$-neighbourhood hypergraph of $H$  \\
		\hline
		
		43. & $H \underset{_{H_r}N[\mathcal{V}]}{\bigvee}H^c$ & ~$\overline{N}$-$r$-complemented neighbourhood hypergraph of $H$ \\
		\hline
		
		44. & $H \underset{_{H_r}N[\mathcal{V}]}{\bigvee}K_n$ &  ~$\overline{N}$-$r$-complete hypergraph of $H$ \\
		\hline	
		
		45. & $H \underset{_{H_r}N[\mathcal{V}]}{\bigvee}K_n^{c}$ & ~$\overline{N}$-$r$-hypergraph of $H$ \\
		\hline
		
		46. & $H \underset{_{H_r}N[\mathcal{V}]}{\bigvee}\overline{H}$ &  ~$\overline{N}$-$r$-total complemented neighbourhood hypergraph of $H$ \\
		\hline
		
		47. & $H \underset{N_{H^c}[\mathcal{V}]}{\bigvee}H$ &  ~$NC$-neighbourhood hypergraph of $H$  \\
		\hline
		
		48. & $H \underset{N_{H^c}[\mathcal{V}]}{\bigvee}H^c$ & ~$NC$-complemented neighbourhood hypergraph of $H$ \\
		\hline
		
		49. & $H \underset{N_{H^c}[\mathcal{V}]}{\bigvee}K_n$ &  ~$NC$-complete hypergraph of $H$ \\
		\hline	
		
		50. & $H \underset{N_{H^c}[\mathcal{V}]}{\bigvee}K_n^{c}$ & ~$NC$-hypergraph of $H$ \\
		\hline
		
		51. & $H \underset{N_{H^c}[\mathcal{V}]}{\bigvee}\overline{H}$ &  ~$NC$-total complemented neighbourhood hypergraph of $H$ \\
		\hline
		
		52. & $H \underset{N_{\overline{H}}[\mathcal{V}]}{\bigvee}H$ &  ~$NTC$-neighbourhood hypergraph of $H$  \\
		\hline
		
		53. & $H \underset{N_{\overline{H}}[\mathcal{V}]}{\bigvee}H^c$ & ~$NTC$-complemented neighbourhood hypergraph of $H$ \\
		\hline
		
		54. & $H \underset{N_{\overline{H}}[\mathcal{V}]}{\bigvee}K_n$ &  ~$NTC$-complete hypergraph of $H$ \\
		\hline	
		
		55. & $H \underset{N_{\overline{H}}[\mathcal{V}]}{\bigvee}K_n^{c}$ & ~$NTC$-hypergraph of $H$ \\
		\hline
		
		56. & $H \underset{N_{\overline{H}}[\mathcal{V}]}{\bigvee}\overline{H}$ &  ~$NTC$-total complemented neighbourhood hypergraph of $H$  \\
		\hline

		57. & $H \underset{	_{H^{c}_r}N[\mathcal{V}]}{\bigvee}H$ &  ~$\overline{N}C$-$r$-neighbourhood hypergraph of $H$  \\
		\hline
		
		58. & $H \underset{_{H^{c}_r}N[\mathcal{V}]}{\bigvee}H^c$ & ~$\overline{N}C$-$r$-complemented neighbourhood hypergraph of $H$\\
		\hline
		
		59. & $H \underset{_{H^{c}_r}N[\mathcal{V}]}{\bigvee}K_n$ &  ~$\overline{N}C$-$r$-complete hypergraph of $H$ \\
		\hline	
		
		60. & $H \underset{_{H^{c}_r}N[\mathcal{V}]}{\bigvee}K_n^{c}$ & ~$\overline{N}C$-$r$-hypergraph of $H$ \\
		\hline
		
		61. & $H \underset{_{H^{c}_r}N[\mathcal{V}]}{\bigvee}\overline{H}$ &  ~$\overline{N}C$-$r$-total complemented neighbourhood hypergraph of $H$ \\
		\hline

		62. & $H \underset{	_{\overline{H}_r}N[\mathcal{V}]}{\bigvee}H$ &  ~$\overline{N}TC$-$r$-neighbourhood hypergraph of $H$ \\
		\hline
		
		63. & $H \underset{_{\overline{H}_r}N[\mathcal{V}]}{\bigvee}H^c$ & ~$\overline{N}TC$-$r$-complemented neighbourhood hypergraph of $H$ \\
		\hline
		
		64. & $H \underset{_{\overline{H}_r}N[\mathcal{V}]}{\bigvee}K_n$ &  ~$\overline{N}TC$-$r$-complete hypergraph of $H$ \\
		\hline	
		
		65. & $H \underset{_{\overline{H}_r}N[\mathcal{V}]}{\bigvee}K_n^{c}$ & ~$\overline{N}TC$-$r$-hypergraph of $H$ \\
		\hline
		
		66. & $H \underset{_{\overline{H}_r}N[\mathcal{V}]}{\bigvee}\overline{H}$ &  ~$\overline{N}TC$-$r$-total complemented neighbourhood hypergraph of $H$ \\
		\hline
		
		67. & $H^c \underset{N_H[\mathcal{V}]}{\bigvee}H^c$ & ~$N$-neighbourhood complement hypergraph of $H$  \\
		\hline
		
		68. & $H^c \underset{N_H[\mathcal{V}]}{\bigvee}K_n$ &  ~$N$-complete complement hypergraph of $H$ \\
		\hline	
		
		69. & $H^c \underset{N_H[\mathcal{V}]}{\bigvee}K_n^{c}$ &  ~$N$-complement hypergraph of $H$  \\
		\hline
		
		70. & $\overline{H} \underset{N_H[\mathcal{V}]}{\bigvee}H^c$ & ~Total $N$-neighbourhood complement hypergraph of $H$ \\
		\hline
		
		71. & $\overline{H} \underset{N_H[\mathcal{V}]}{\bigvee}K_n$ & ~$N$-complete total complement hypergraph of $H$ \\
		\hline	
		
		72. & $\overline{H} \underset{N_H[\mathcal{V}]}{\bigvee}K_n^{c}$ & ~$N$-total complement hypergraph of $H$ \\
		\hline
		
		73. & $\overline{H} \underset{N_H[\mathcal{V}]}{\bigvee}\overline{H}$ & ~$N$-neighbourhood total complement hypergraph of $H$  \\
		\hline
		
		74. & $H^c \underset{N_{\overline{H}}[\mathcal{V}]}{\bigvee}H^c$ & ~$NTC$-neighbourhood-complement hypergraph of $H$ \\
		\hline
		
		75. & $H^c \underset{N_{\overline{H}}[\mathcal{V}]}{\bigvee}K_n$ &  ~$NTC$-complete-complement hypergraph of $H$ \\
		\hline	
		
		76. & $H^c \underset{N_{\overline{H}}[\mathcal{V}]}{\bigvee}K_n^{c}$ &  ~$NTC$-complement hypergraph of $H$ \\
		\hline
		
		77. & $\overline{H} \underset{N_{\overline{H}}[\mathcal{V}]}{\bigvee}H^c$ & ~Total $NTC$-neighbourhood-complement hypergraph of $H$ \\
		\hline
		
		78. & $\overline{H} \underset{N_{\overline{H}}[\mathcal{V}]}{\bigvee}K_n$ & ~$NTC$-complete total complement hypergraph of $H$  \\
		\hline	
		
		79. & $\overline{H} \underset{N_{\overline{H}}[\mathcal{V}]}{\bigvee}K_n^{c}$ & ~$NTC$-total complement hypergraph of $H$ \\
		\hline
		
		80. & $\overline{H} \underset{N_{\overline{H}}[\mathcal{V}]}{\bigvee}\overline{H}$ & ~$NTC$-neighbourhood total complement hypergraph of $H$   \\
		\hline
		
		81. & $H^c \underset{N_{H^c}[\mathcal{V}]}{\bigvee}H^c$ & ~$NC$-neighbourhood-complement hypergraph of $H$   \\
		\hline
		
		82. & $H^c \underset{N_{H^c}[\mathcal{V}]}{\bigvee}K_n$ &  ~$NC$-complete-complement hypergraph of $H$  \\
		\hline	
		
		83. & $H^c \underset{N_{H^c}[\mathcal{V}]}{\bigvee}K_n^{c}$ &  ~$NC$-complement hypergraph of $H$  \\
		\hline
		
		84. & $\overline{H} \underset{N_{H^c}[\mathcal{V}]}{\bigvee}H^c$ &  ~Total $NC$-neighbourhood-complement hypergraph of $H$ \\
		\hline
		
		85. & $\overline{H} \underset{N_{H^c}[\mathcal{V}]}{\bigvee}K_n$ & ~$NC$-complete total complement hypergraph of $H$ \\
		\hline	
		
		86. & $\overline{H} \underset{N_{H^c}[\mathcal{V}]}{\bigvee}K_n^{c}$ & ~$NC$-total complement hypergraph of $H$ \\
		\hline
		
		87. & $\overline{H} \underset{N_{H^c}[\mathcal{V}]}{\bigvee}\overline{H}$ & ~$NC$-neighbourhood total complement hypergraph of $H$  \\
		\hline
		
		88. & $H^c \underset{_{H_r}N[\mathcal{V}]}{\bigvee}H^c$ & ~$\overline{N}$-$r$-neighbourhood-complement hypergraph of $H$   \\
		\hline
		
		89. & $H^c \underset{_{H_r}N[\mathcal{V}]}{\bigvee}K_n$ & ~$\overline{N}$-$r$-complete-complement hypergraph of $H$  \\
		\hline
		
		90. & $H^c \underset{_{H_r}N[\mathcal{V}]}{\bigvee}K_n^{c}$ & ~$\overline{N}$-$r$-complement hypergraph of $H$  \\
		\hline
		
		91. & $\overline{H} \underset{_{H_r}N[\mathcal{V}]}{\bigvee}H^c$ &  ~Total $\overline{N}$-$r$-neighbourhood complement hypergraph of $H$  \\
		\hline
		
		92. & $\overline{H} \underset{_{H_r}N[\mathcal{V}]}{\bigvee}K_n$ &~$\overline{N}$-$r$-complete total complement hypergraph of $H$  \\
		\hline
		
		93. & $\overline{H} \underset{_{H_r}N[\mathcal{V}]}{\bigvee}K_n^{c}$ & ~$\overline{N}$-$r$-total complement hypergraph of $H$   \\
		\hline
		
		94. & $\overline{H} \underset{_{H_r}N[\mathcal{V}]}{\bigvee}\overline{H}$ & ~$\overline{N}$-$r$-neighbourhood total complement hypergraph of $H$\\
		\hline
		
		95. & $H^c \underset{_{H^{c}_r}N[\mathcal{V}]}{\bigvee}H^c$ & ~$\overline{N}C$-$r$-neighbourhood-complement hypergraph of $H$   \\
		\hline
		
		96. & $H^c \underset{_{H^{c}_r}N[\mathcal{V}]}{\bigvee}K_n$ & ~$\overline{N}C$-$r$-complete-complement hypergraph of $H$ \\
		\hline
		
		97. & $H^c \underset{_{H^{c}_r}N[\mathcal{V}]}{\bigvee}K_n^{c}$ & ~$\overline{N}C$-$r$-complement hypergraph of $H$   \\
		\hline
		
		98. & $\overline{H} \underset{_{H^{c}_r}N[\mathcal{V}]}{\bigvee}H^c$ &  ~Total $\overline{N}C$-$r$-neighbourhood-complement hypergraph of $H$ \\
		\hline
		
		99. & $\overline{H} \underset{_{H^{c}_r}N[\mathcal{V}]}{\bigvee}K_n$ &~$\overline{N}C$-$r$-complete-total complement hypergraph of $H$   \\
		\hline
		
		100. & $\overline{H} \underset{_{H^{c}_r}N[\mathcal{V}]}{\bigvee}K_n^{c}$ & ~$\overline{N}C$-$r$-total complement hypergraph of $H$   \\
		\hline
		
		101. & $\overline{H} \underset{_{H^{c}_r}N[\mathcal{V}]}{\bigvee}\overline{H}$ & ~$\overline{N}C$-$r$-neighbourhood total complement hypergraph of $H$  \\
		\hline
		
		102. & $H^c \underset{_{\overline{H}_r}N[\mathcal{V}]}{\bigvee}H^c$ & ~$\overline{N}TC$-$r$-neighbourhood-complement hypergraph of $H$  \\
		\hline
		
		103. & $H^c \underset{_{\overline{H}_r}N[\mathcal{V}]}{\bigvee}K_n$ & ~$\overline{N}TC$-$r$-complete-complement hypergraph of $H$  \\
		\hline
		
		104. & $H^c \underset{_{\overline{H}_r}N[\mathcal{V}]}{\bigvee}K_n^{c}$ & ~$\overline{N}TC$-$r$-complement hypergraph of $H$   \\
		\hline
		
		105. & $\overline{H} \underset{_{\overline{H}_r}N[\mathcal{V}]}{\bigvee}H^c$ &  ~Total $\overline{N}TC$-$r$-neighbourhood-complement hypergraph of $H$ \\
		\hline
		
		106. & $\overline{H} \underset{_{\overline{H}_r}N[\mathcal{V}]}{\bigvee}K_n$ &~$\overline{N}TC$-$r$-complete total complement hypergraph of $H$  \\
		\hline
		
		107. & $\overline{H} \underset{_{\overline{H}_r}N[\mathcal{V}]}{\bigvee}K_n^{c}$ & ~$\overline{N}TC$-$r$-total complement hypergraph of $H$  \\
		\hline
		
		108. & $\overline{H} \underset{_{\overline{H}_r}N[\mathcal{V}]}{\bigvee}\overline{H}$ & ~$\overline{N}TC$-$r$-neighbourhood-total complement hypergraph of $H$ \\
		\hline
		
		109. & $K_n^c \underset{N_H[\mathcal{V}]}{\bigvee}K_n^{c}$ &  ~Duplicate hypergraph of $H$  \\
		\hline
		
		110. & $K_n \underset{N_H[\mathcal{V}]}{\bigvee}K_n^{c}$ &  ~Duplicate complete hypergraph of $H$  \\
		\hline

		111. & $K_n \underset{N_H[\mathcal{V}]}{\bigvee}K_n$ &  ~Fully complete duplicate hypergraph of $H$ \\
		\hline
		
		112. & $K_n^c \underset{_{H_r}N[\mathcal{V}]}{\bigvee}K_n^{c}$ &   ~$r$-$D\overline{N}$-hypergraph of $H$  \\
		\hline
		
		113. & $K_n \underset{_{H_r}N[\mathcal{V}]}{\bigvee}K_n^{c}$ &  ~$r$-$D\overline{N}$-complete hypergraph of $H$  \\
		\hline

		114. & $K_n \underset{_{H_r}N[\mathcal{V}]}{\bigvee}K_n$ &  ~Fully complete $r$-$D\overline{N}$-hypergraph of $H$ \\
		\hline
		
		115. & $K_n^c \underset{N_{H^{c}}[\mathcal{V}]}{\bigvee}K_n^{c}$ &  ~Complemented duplicate hypergraph of $H$  \\
		\hline
		
		116. & $K_n \underset{N_{H^c}[\mathcal{V}]}{\bigvee}K_n^{c}$ &  ~Complemented duplicate complete hypergraph of $H$  \\
		\hline

		117. & $K_n \underset{N_{H^c}[\mathcal{V}]}{\bigvee}K_n$ &  ~Fully complete complemented duplicate hypergraph of $H$  \\
		\hline
		
		118. & $K_n^c \underset{N_{\overline{H}}[\mathcal{V}]}{\bigvee}K_n^{c}$ &  ~Total complemented duplicate hypergraph of $H$  \\
		\hline
		
		119. & $K_n \underset{N_{\overline{H}}[\mathcal{V}]}{\bigvee}K_n^{c}$ &  ~Total complemented duplicate complete hypergraph of $H$  \\
		\hline

		120. & $K_n \underset{N_{\overline{H}}[\mathcal{V}]}{\bigvee}K_n$ &  ~Fully complete total complemented duplicate hypergraph of $H$  \\
		\hline
		
		121. & $K_n^c \underset{_{\overline{H}_r}N[\mathcal{V}]}{\bigvee}K_n^{c}$ &  ~Closed duplicate $r$-total complemented hypergraph of $H$ \\
		\hline
		
		122. & $K_n \underset{_{\overline{H}_r}N[\mathcal{V}]}{\bigvee}K_n^{c}$ &  ~Closed duplicate complete $r$-total complemented hypergraph of $H$  \\
		\hline

		123. & $K_n \underset{_{\overline{H}_r}N[\mathcal{V}]}{\bigvee}K_n$ & ~Fully complete closed duplicate $r$-total complemented hypergraph of $H$  \\
		\hline
		
		124. & $K_n^c \underset{_{H^c_r}N[\mathcal{V}]}{\bigvee}K_n^{c}$ &  ~Closed $r$-duplicate hypergraph of $H$  \\
		\hline
		
		125. & $K_n \underset{_{H^c_r}N[\mathcal{V}]}{\bigvee}K_n^{c}$ &  ~Closed $r$-duplicate complete hypergraph of $H$  \\
		\hline

		126. & $K_n \underset{_{H^c_r}N[\mathcal{V}]}{\bigvee}K_n$ &  ~Fully complete closed $r$-duplicate hypergraph of $H$  \\
		\hline
		
		
		%
		%
		%
		\caption{New unary hypergraph operations defined as tensor join of two hypergraphs}\label{utab}
	\end{longtable}	
	When $r=1$, the hypergraph given in S.No.1 of Table~\ref{utab} becomes $H \underset{I[\mathcal{V}]}{\bigvee}H$ and we call it simply as \textit{the mirror hypergraph of $H$}. Similarly, the rest of the hypergraph operations defined in Table~\ref{utab} in which $_rT[\mathcal{V}]$ is involved can be renamed.
	
	Now, we show that the hypergraph operations listed in Table~\ref{Utab} are unary. Consider the $r$-Mirror hypergraph of $H$. It is constructed from the hypergraph $H$ as follows: First take $H$ and corresponds to each of its vertex, add a new vertex. Now, make each set $S$ of new vertices as an edge in the $r$-Mirror hypergraph of $H$ if and only if the set of vertices in $H$ corresponding to the vertices of $S$ forms an edge in $H$. Then for each $r$-subset $S_r$ of vertices of $H$, make the set of all vertices in $S_r$ together with all the new vertices corresponding to each vertices in $S_r$ as an edge in the $r$-Mirror hypergraph of $H$. The resulting hypergraph is the desired one. Similarly, the rest of the operations can be viewed.
	\subsection{Some unary hypergraph operations as $\mathcal{T}^*$-join of hypergraphs in $\mathcal{G}$}
	Let $G_i(V_i,E_i), i=1,2,\dots,k$ be $k(>1)$ copies of a hypergraph $H$ with $|V(H)|=n$. 
	\normalfont
	Let $\mathcal{G}=(G_i(V_i,E_i,W_i))_{i=1}^k$. For each $S\in\widehat{[k]}$, let $\mathcal{V}_S=(V_i)_{i\in S}$. Let $\mathcal{T}^*=\{T^*[\mathcal{V}_S]~|~S\in \widehat{[k]}\}$ be a set of indicating tensors of type-2. In Table~\ref{Tabb}, we list some new classes of unary hypergraph operations as  $\mathcal{T}^*$-join of hypergraphs in $\mathcal{G}$, for some suitable $\mathcal{T}^*$ as mentioned in the same table.  In this table, we take $1<l\leq k$ and $\mathbf{0}$ denotes a zero tensor of appropriate order and dimension. 
	\setlength\extrarowheight{1pt}
	\renewcommand{\arraystretch}{1.0}
	\tabcolsep=0.11cm
	\begin{longtable}{ | M{1.2cm}| M{6cm} |M{4cm}| }
		\hline
		\textbf{S. No.}  & \textbf{Name of the hypergraph}  &
		$T^*[\mathcal{V}_S]$\\ 
		\hline
		
		1. & $(l,r)$-mirror hypergraph of $H$	&
		$\begin{cases}
			_rT[\mathcal{V}_S]&\text{if}~|S|=l,\\
			\mathbf{0}&\text{otherwise}
		\end{cases}$
		\\
		\hline
		
		2.& Join $l$-neighbourhood hypergraph of $H$ & 
		$\begin{cases}
			J[\mathcal{V}_S] &\text{if}~|S|=l,\\ \mathbf{0} &\text{otherwise}
		\end{cases}$
		\\
		\hline
		
		3. & $VC$-$(l,r)$-neighbourhood hypergraph of $H$& $\begin{cases}
			\Im_r[\mathcal{V}_S]&\text{if}~|S|=l,\\
			\mathbf{0} &\text{otherwise}
		\end{cases}$  \\
		\hline
		
		%
		\caption{Viewing some new unary hypergraph operations as $\mathcal{T}^*$-join of hypergraphs in $\mathcal{G}$.} \label{Tabb} \end{longtable}
	\subsection{Some classes of hypergraphs as $(H,\mathcal{T})$-join of hypergraphs}
	
	Whenever we consider the $(H,\mathcal{T})$-weighted/unweighted join of weighted/unweighted hypergraphs, without loss of generality, we take the vertex set of $H$ of cardinality $k$ as $[k].$
	
	Let $H$ be a hypergraph with $|V(H)|=k$ and let $\mathcal{G}=(G_i(V_i,E_i))_{i=1}^k$ be a sequence of $k$ hypergraphs. For each $e\in E(H)$, let $\mathcal{V}_e=(V_i)_{i\in e}$.
	In Table~\ref{Tab4}, we list some classes of hypergraphs that can be viewed as a $(H, \mathcal{T})$-join of hypergraphs in $\mathcal{G}$, for some suitable $H$, $G_i$ and $\mathcal{T}$.  
	\setlength\extrarowheight{1pt}
\renewcommand{\arraystretch}{1.0}
\tabcolsep=0.11cm
\begin{longtable}{ | M{.61cm}| M{4.5cm} | M{.8cm} |M{.8cm}| M{4cm}| }
	\hline
	\textbf{S. No.}  & \textbf{Name of the hypergraph}  & $H$ &   $G_i$ &
	$\mathcal{T}$\\ 
	\hline
	
	1. & Join of set $\mathcal{G}$ of $m$-uniform hypergraphs on a backbone hypergraph $H$, $r(H)\leq m$  \cite{sarkar2020joins} & $H$	& $G_i$  & $\{{}_{B_e}T[\mathcal{V}_e]~|~e\in E(H)\}$, where $B_e=\{m\}$.
	\\
	\hline
	
	2.& Join of set $\mathcal{G}$ of non-uniform hypergraphs on a backbone hypergraph $H$ \cite{sarkar2020joins} & $H$ & $G_i$ & $\{{}_{B_e}T[\mathcal{V}_e]~|~e\in E(H)\}$, where $B_e\subseteq \{|e|,|e|+1,\dots,N_e\}$. 
	\\
	\hline
	
	3. & Complete $m$-uniform strong  $k$-partite hypergraph $(k\geq m)$\cite{sarkar2020joins} & $K_k^m$ & $K_{n_i}^c$  & $\{{}_{B_e}T[\mathcal{V}_e]~|~e\in E(H)\}$, where $B_e=\{m\}$.  \\
	\hline
	
	4. & Complete strong $k$-partite hypergraph & $K_k$ & $K_{n_i}^c$  &$\{{}_{B_e}T[\mathcal{V}_e]~|~e\in E(H)\}$, where $B_e=\{|e|\}$. \\
	\hline
	
	5. & Lexicographic product of the hypergraphs $H_1$ and $H_2$ \cite{hellmuth2012survey} & $H_1$  & $H_2$ &  $\{{}_{B_e}T[\mathcal{V}_e]~|~e\in E(H)\}$, where $B_e=\{|e|\}$. \\
	\hline
	
	6. & Cartesian product of the hypergraphs $H_1$ and $H_2$ \cite{banerjee2021spectrum} & $H_1$  & $H_2$ & $\{I[\mathcal{V}_e]~|~e\in E(H)\}$\\
	\hline
	\caption{Viewing some existing and new class of hypergraphs as $(H,\mathcal{T})$-join of hypergraphs.} \label{Tab4} \end{longtable}
	\section{Spectra of the tensor join of weighted hypergraphs}\label{Sec4}
	In this section, we obtain the characteristic polynomial of the adjacency, the Laplacian, the normalized Laplacian matrices of some classes of hypergraphs constructed by the tensor join operations defined in Section~\ref{defn sec}. For the computation of the normalized Laplacian spectrum of the tensor join of hypergraphs, it is assumed that the constituting hypergraphs  do not have isolated vertices.
	\subsection{Spectra of the $T[\mathcal{A}]$-join of hypergraphs}	
	Let $\mathcal{G}=(G_i(V_i,E_i,W_i))_{i=1}^k$ be a sequence of $k$
	weighted hypergraphs. Consider an indicating tensor $T[\mathcal{V}]$, where $\mathcal{V}=(V_i)_{i=1}^k$. We construct the hypergraph $\underset{T[\mathcal{V}]}{\bigvee}\mathcal{G}(V,E,W)$ with a weight function $W:E\rightarrow \mathbb{R}_{\geq 0}$ defined by,	
	\begin{align}\label{wt}
		W(e)=\begin{cases}
			W_i(e)&\text{if}~ e\in E_i;\\
			~~w_c &\text{if}~ e\notin E_i ~\text{with}~|e|=c,~ \text{for}~i=1,2,\dots,k.
		\end{cases},
	\end{align} 
	where $w_c$ is a non-negative real number corresponding to a new edge of cardinality $c$.
	
	Throughout this section, we consider the weight function as defined above for any $T[\mathcal{V}]$-join of weighted hypergraphs, unless, we specifically mentioned otherwise.

	\begin{thm}\label{t2.2}
		Let $\mathcal{G}=(G_i(V_i,E_i,W_i))_{i=1}^k$ be a sequence of weighted $r_i$-regular hypergraphs $G_i$ with $|V_i|=n_i$, let $\mathcal{V}= (V_i)_{i=1}^k$ and let  $X=\{2,3,\dots,N\}.$ Consider an indicating tensor $T[\mathcal{V}]$ such that for every $p\in V_i$, $q\in V_j$ and $c\in X$, $|E_{p,q}^{c}(T[\mathcal{V}])|$ is a constant, say $n_{ij}^{(c)}$ for all $1\leq i\leq j\leq k$. Then the characteristic polynomial of the adjacency (resp. the Laplacian, the normalized Laplacian) matrix of the weighted hypergraph $\underset{T[\mathcal{V}]}{\bigvee}\mathcal{G}$ is
		\begin{center}
			$\left\{\underset{i=1}{\overset{k}{\prod}}~\underset{j=1;j\neq i}{\overset{n_i}{\prod}}(x-\alpha_i\lambda_{ij}-\beta_i)\right\}\times P_R(x)$,
		\end{center} where $\lambda_{ij}$ is a non-Perron adjacency eigenvalue of $G_i$ for $i=1,2,\dots,k$; $j=1,2,\dots,n_i$ and 
		
		%
		\[R=
		\begin{bmatrix}
			r_1\alpha_1+\beta_1+n_1\gamma_1& n_2\delta_{12}\displaystyle\underset{c\in X}{ \sum}\frac{w_c\cdot n_{12}^{(c)}}{c-1} &\cdots &n_k\delta_{1k}\displaystyle\underset{c\in X}{ \sum}\frac{w_c\cdot n_{1k}^{(c)}}{c-1} \\
			n_1\delta_{12}\displaystyle\underset{c\in X}{ \sum}\frac{w_c\cdot n_{12}^{(c)}}{c-1}& r_2\alpha_2+\beta_2+n_2\gamma_2& \cdots & 	n_k\delta_{2k}\displaystyle\underset{c\in X}{ \sum}\frac{w_c\cdot n_{2k}^{(c)}}{c-1}\\
			\vdots & \vdots & \ddots& \vdots  \\
			n_1\delta_{1k}\displaystyle\underset{c\in X}{ \sum}\frac{w_c\cdot n_{1k}^{(c)}}{c-1}& n_2\delta_{2k}\displaystyle\underset{c\in X}{ \sum}\frac{w_c\cdot n_{2k}^{(c)}}{c-1} & \cdots & r_k\alpha_k+\beta_k+n_k\gamma_k	 
		\end{bmatrix}_{k\times k}
		\]
		and	for $1\leq i\leq j\leq k$, the values $\alpha_i, ~\beta_i,~\gamma_i,~\delta_{ij}$ are given in Table~\ref{values table1} corresponding to the respective matrices, where
		$$z_i=r_i+(n_i-1)\sum\limits_{c\in X}\frac{n_{ii}^{(c)}\cdot w_c}{c-1}+\sum\limits_{j=1,j\neq i}^{k}n_j\sum\limits_{c\in X}\frac{n_{ij}^{(c)}\cdot w_c}{c-1}~\text{with}~ n_{ij}^{(c)}=n_{ji}^{(c)}~\text{for}~ i,j=1,2,\dots,k.$$
		\setlength\extrarowheight{1pt}
		\renewcommand{\arraystretch}{1.0}
		\tabcolsep=0.11cm
		\begin{center}
			\begin{longtable}{ | M{4cm} | M{1cm} |M{2.5cm} |M{1cm} |M{1cm}| }
				\hline
				\textbf{Name of the matrix} & $\alpha_i$ & $\gamma_i$ & $\beta_i$ & $\delta_{ij}$ \\ \hline
				Adjacency matrix & $1$ & 	$\sum\limits_{c\in X}\frac{w_c\cdot n_{ii}^{(c)}}{c-1}$ & $-\gamma_i$ & $1$\\
				\hline
				Laplacian matrix & $-1$ & $-\sum\limits_{c\in X}\frac{w_c\cdot n_{ii}^{(c)}}{c-1}$ & $z_i-\gamma_i$ & $-1$ \\
				\hline 
				Normalized Laplacian matrix & $\frac{-1}{z_i}$ & $\alpha_i\sum\limits_{c\in X}\frac{w_c\cdot n_{ii}^{(c)}}{c-1}$ & $1-\gamma_i$ & $\frac{-1}{\sqrt{z_iz_j}}$\\
				\hline\caption{Necessary values to compute the spectrum of the matrices associated with $\underset{T[\mathcal{V}]}{\bigvee}\mathcal{G}$}\label{values table1}
			\end{longtable}	
		\end{center}
	\end{thm}
	
	\begin{proof}
		The adjacency (resp. the Laplacian, the normalized Laplacian) matrix of $\underset{T[\mathcal{V}]}{\bigvee}\mathcal{G}$ is a $k\times k$ symmetric block matrix of order $N\times N$ in which the $(i,i)^{th}$ block is $$\alpha_iA(G_i)+\beta_iI_{n_i}+\gamma_iJ_{n_i}$$ and for $i\neq j$, the $(i,j)^{th}$ block is $$\delta_{ij}\displaystyle\underset{c\in X}{ \sum}\frac{w_c\cdot n_{ij}^{(c)}}{c-1}J_{n_i\times n_j},$$ 
		where $N:=\sum\limits_{i=1}^kn_i$ and the values $\alpha_i, \beta_i, \gamma_i, \delta_{ij}$ for $i,j=1,2,\dots,k$ are given in Table~\ref{values table1}

		Notice that for each $i=1,2,\dots,k$, the adjacency matrix $A(G_i)$ of $G_i$ is real symmetric of order $n_i$ with the constant row sum $r_i$. Thus each $A(G_i)$ has an orthogonal basis of $\mathbb{R}^{n_i}$ consisting of its eigenvectors, including the all-one vector $J_{n_i\times 1}$ corresponds to the eigenvalue $r_i$.
		Let us denote the eigenvectors of $A(G_i)$ by $X_{i1}(=J_{n_{i\times 1}})
		,X_{i2},\cdots,X_{in_{i}}$ corresponds to the eigenvalues $\lambda_{i1}(=r_i), \lambda_{i2},\cdots,\lambda_{in_i},$ for all $i=1,2,\dots,n_i$. Let $$\mathcal{X}_{\mathbf{ij}}:=[\mathbf{0},\mathbf{0},\cdots,\underset{\text{i-th~place}}{\underbrace{X_{ij}}},\mathbf{0},\cdots\mathbf{0}]_{\underset{\mathbf{1\times N}}{}}^{\overset{\mathbf{T}}{}}$$ for all $i=1,2,\dots,k$, $j=2,\dots,n_i$.
		Then for each $i=1,2,\dots,k$, $j=2,\dots,n_i$,  $\alpha_i\lambda_{ij}+\beta_i$ is an eigenvalue of $\mathcal{A}$ corresponds to the eigenvector $\mathcal{X}_{\mathbf{ij}}$.
		%
		%
		%
		Since, the span of the remaining $k$ eigenvectors of $\mathcal{A}$ is same as the span of vectors $$[\mathbf{0},\mathbf{0},\dots,\underset{\text{i-th~place}}{\underbrace{J_{n_i\times 1}}},\mathbf{0},\dots,\mathbf{0}]_{\underset{\mathbf{1\times N}}{}}^{\overset{\mathbf{T}}{}}, i=1,2,\dots,k,$$ so let $\mu$ be an eigenvalue of $\mathcal{A}$ corresponds to the eigenvector 
		$$\mathcal{Y}=[a_1J_{n_1\times 1},a_2J_{n_2\times 1},\cdots,a_kJ_{n_k\times 1}],$$ where $(a_1,a_2,\dots, a_k)$ is a non-zero vector in $\mathbb{R}^k.$
		Then the system of equations $(\mathcal{A}-\mu)\mathcal{Y}=\mathbf{0}$ reduces to the system of equations $(R-\mu)y'=\mathbf{0}$, where $y'=(a_1,a_2,\dots, a_k)$ and the matrix $R$ is as mentioned in the statement of this theorem. Thus the remaining eigenvalues of $\mathcal{A}$ are the eigenvalues of the matrix $R$. This completes the proof.		
	\end{proof}	

	In the following corollary, we establish infinite families of cospectral hypergraphs by using the $T[\mathcal{A}]$-join operation on hypergraphs.
	\begin{cor}
		\normalfont
		Let $G_i(V_i,E_i,W_i)$ and $G_i'(V_i',E_i',W_i')$ be $r_i$-regular $A-$cospectral weighted hypergraphs for $i=1,2,\dots, k$. Let $\mathcal{G}=(G_i(V_i,E_i,W_i))_{i=1}^k$, $\mathcal{G}'=(G_i'(V_i',E_i',W_i'))_{i=1}^{k}$ and let $X=\{2,3,\dots,N\}$.  Let $\mathcal{V}=(V_i)_{i=1}^k$ and $\mathcal{V}'=(V_i')_{i=1}^k$. Consider an indicating tensor $T[\mathcal{V}]$ such that for every $p\in V_i$ and $q\in V_j$, $|E_{p,q}^{c}(T[\mathcal{V}])|=n_{ij}^{(c)}$ for all $c\in X,~1\leq i\leq j\leq k$. Let $T[\mathcal{V}']$ be an indicating tensor such that $T[\mathcal{V}']=T[\mathcal{V}]$. Then the weighted $T[\mathcal{V}]$-join of hypergraphs in $\mathcal{G}$ and the weighted $T[{\mathcal{V}'}]$-join of hypergraphs in $\mathcal{G}'$ are simultaneously $A-$cospectral, $L-$cospectral and $\mathcal{L}-$cospectral.
	\end{cor}
	\begin{proof}
		Since $G_i$ and $G_i'$ are $r_i$-regular and have the same adjaceny spectrum, the result directly follows from Theorem~\ref{t2.2}.
	\end{proof}
	\begin{cor}\label{c2.1}
		\normalfont
		Let $\mathcal{G}=(G_i(V_i,E_i,W_i))_{i=1}^k$, where $G_i$ is a weighted $r_i$-regular $m$-uniform hypergraph with $|V_i|=n_i$ for $i=1,2,\dots,k$. Let $\mathcal{V}=(V_i)_{i=1}^k$. Consider an indicating tensor $T[\mathcal{V};m]$ such that for every $p\in V_i$ and $q\in V_j$, $|E_{p,q}^{m}(T[\mathcal{V};m])|$ is a constant, say $n_{ij}^{(m)}$, for $1\leq i\leq j\leq k$.
		Then the characteristic polynomial of the adjacency (resp. the Laplacian, the normalized Laplacian) matrix of the weighted hypergraph $\underset{T[\mathcal{V};m]}{\bigvee}\mathcal{G}$ is
		\begin{center}
			$\left\{\underset{i=1}{\overset{k}{\prod}}~\underset{j=1;j\neq i}{\overset{n_i}{\prod}}(x-\alpha_i\lambda_{ij}-\beta_i)\right\}\times P_R(x)$,
		\end{center}where $\lambda_{ij}$ is a non-Perron adjacency eigenvalue of $G_i$ for $i=1,2,\dots,k;~j=1,2,\dots,n_i$ and 
		\[R=
		\begin{bmatrix}
			r_1\alpha_1+\beta_1+n_1\gamma_1& \delta_{12}\cdot n_2\frac{w_m\cdot n_{12}^{(m)}}{m-1} &\cdots &\delta_{1k}\cdot n_k\frac{w_m\cdot n_{1k}^{(m)}}{m-1} \\
			\delta_{12}\cdot n_1\frac{w_m\cdot n_{12}^{(m)}}{m-1}& r_2\alpha_2+\beta_2+n_2\gamma_2& \cdots & 	\delta_{2k}\cdot n_k\frac{w_m\cdot n_{2k}^{(m)}}{m-1} \\
			\vdots & \vdots & \ddots& \vdots  \\
			\delta_{1k}\cdot n_1\frac{w_m\cdot n_{1k}^{(m)}}{m-1}& \delta_{2k}\cdot n_2\frac{w_m\cdot n_{2k}^{(m)}}{m-1}& \cdots & r_k\alpha_k+\beta_k+n_k\gamma_k	 
		\end{bmatrix}_{k\times k}
		\]
		and for $1\leq i\leq j\leq k$, the values $\alpha_i, ~\beta_i,~\gamma_i,~\delta_{ij}$ are given in Table~\ref{values table2}; $$z_i=r_i+(n_i-1)\frac{n_{ii}^{(m)}w_m}{m-1}+\sum\limits_{j=1,j\neq i}^{k}n_j\frac{n_{ij}^{(m)}w_m}{m-1}$$	with $n_{ij}^{(m)}=n_{ji}^{(m)}$ for $i,j=1,2,\dots,k.$\\
		\setlength\extrarowheight{1pt}
		\renewcommand{\arraystretch}{1.0}
		\tabcolsep=0.11cm
		\begin{longtable}{ | M{4cm} | M{1cm} |M{2cm} |M{1.5cm} |M{1cm}| }
			\hline
			\textbf{Name of the matrix} & $\alpha_i$ & $\gamma_i$ & $\beta_i$ & $\delta_{ij}$ \\ \hline
			Adjacency matrix & $1$ & 	$\frac{w_m\cdot n_{ii}^{(m)}}{m-1}$ & $-\gamma_i$ & $1$\\
			\hline
			Laplacian matrix & $-1$ & $-\frac{w_m\cdot n_{ii}^{(m)}}{m-1}$ & $z_i-\gamma_i$ & $-1$ \\
			\hline 
			Normalized Laplacian matrix& $\frac{-1}{z_i}$ & $\alpha_i\frac{w_m\cdot n_{ii}^{(m)}}{m-1}$ & $1-\gamma_i$ & $\frac{-1}{\sqrt{z_iz_j}}$\\
			\hline\caption{Necessary values to compute the spectrum of the matrices associated with $\underset{T[\mathcal{V};m]}{\bigvee}\mathcal{G}$.}\label{values table2}
		\end{longtable}
	\end{cor}
	\begin{proof}
		If we take $X=\{m\}$ in Theorem~\ref{t2.2}, then 
		
		$$z_i= r_i+(n_i-1)\frac{n_{ii}^{(m)}\cdot w_m}{m-1}+\sum\limits_{j=1,j\neq i}^{k}n_j\frac{n_{ij}^{(m)}\cdot w_m}{m-1}$$ 
				for all $i=1,2,\dots,k$ and so the proof follows.
	\end{proof}
	\begin{notation}
		\normalfont
		Let $S$ be a family of $k$ finite sets $A_1, A_2,\dots, A_k$ and let $c\in\{2,3,\dots,|A_1|+|A_2|+\cdots+|A_k|\}$. For $1\leq i\leq j\leq k$, we denote,
		\[
		n_{ij}^{c}(S)=
		\begin{cases}
			\displaystyle\underset{l_1+l_2+\cdots+ l_k=c-2}{\underset{l_i\geq0,~l_t>0~(t\neq i)}{\sum}}\left(^{|A_1|}_{~~l_1}\right)\left(^{|A_2|}_{~~l_2}\right)\cdots\left(^{|A_{i-1}|}_{~~l_{i-1}}\right)\left(^{|A_i|-2}_{~~~~l_i}\right) \cdots \left(^{|A_k|}_{~~l_k}\right)&\text{if}~i=j;\\
			\displaystyle\underset{l_1+l_2+\cdots+ l_k=c-2}{\underset{l_i,l_j\geq0,~l_t>0~( t\neq i,j)}{\sum}}\left(^{|A_1|}_{~~l_1}\right)\cdots\left(^{|A_{i-1}|}_{~~l_{i-1}}\right) \left(^{|A_i|-1}_{~~~l_i}\right)\cdots\left(^{|A_{j-1}|}_{~~l_{j-1}}\right)\left(^{|A_j|-1}_{~~~l_j}\right)\cdots \left(^{|A_k|}_{~~l_k}\right)&\text{if}~i\neq j;\\
			0&\text{otherwise.}
		\end{cases}	
		\]
	\end{notation}
	\begin{cor}
		\normalfont
		Assume additionally that
		the hypergraphs given in S.Nos. 4 and 5 of Table~\ref{tab3} be constructed by $r_i$-regular weighted hypergraph $H_i(V_i,E_i,W_i)$ for all $i=1,2,\dots,k$. 
		Then the characteristic polynomial of the adjacency (resp. the Laplacian, the normalized Laplacian) matrix of the weighted hypergraphs given in Table~\ref{tab3} are obtained from Theorem~\ref{t2.2} by taking the values $\alpha_i, \beta_i,\gamma_i,\delta_i$ as in that theorem, the values $n_{ij}^{(c)}$, $r_i$ as given in Table~\ref{tab4} and taking $X=B$ given in Table~\ref{tab3} for the respective hypergraph.\\	
		\setlength\extrarowheight{1pt}
	\renewcommand{\arraystretch}{1.0}
	\tabcolsep=0.11cm
	
	\begin{longtable}{ | M{.5cm} | M{4.5cm} |M{.5cm} |M{.5cm} |M{4.5cm}| }
		\hline
		\textbf{S. No.}  & \textbf{Name of the hypergraph}  & $r_i$ & $c$ & 
		$n_{ij}^{(c)}$ \\ 
		\hline							1. & Complete $m$-uniform $m$-partite hypergraph $ \cite{sarkar2020joins}$  & $0$ & 
		$m$  &$n_{ij}^{(m)}=\begin{cases}
		~~~~~0~~~~~~~~~~\text{if}~ i=j\\
		\overset{m}{\underset{p=1, p\neq i,j}{\prod}}n_p	~~\text{if}~ i\neq j
		\end{cases}$ \\
		\hline 				

		
		2. &Complete $m$-uniform weak $k$-partite hypergraph, $k\leq m$  \cite[Example 3.1.2]{sarkar2020joins} & $0$ & $m$ &  $n_{ij}^{m}(S)$, where $S=\{V(K_{n_i}^c)\}_{i=1}^k$.   \\
		\hline
		3. & Complete weak $k$-partite hypergraph & $0$ & $c$  & $n_{ij}^{c}(S)$, where $S=\{V(K_{n_i}^c)\}_{i=1}^k$.    \\
		\hline
		4. & Join of a collection $\mathcal{G}$ of non-uniform hypergraphs \cite[Theorem 3.2.1]{sarkar2020joins} & $r_i$ & $c$ & $n_{ij}^{c}(S)$, where $S=\{V_i\}_{i=1}^k$.  \\
		\hline
		
		5. & Join of a collection $\mathcal{G}$ of $m$-uniform hypergraphs \cite{sarkar2020joins} & $r_i$  &  $m$ &   $n_{ij}^{m}(S)$, where $S=\{V_i\}_{i=1}^k$.  \\
		\hline	
		\caption{Necessary values for determining the spectrum of the matrices associated with the hypergraphs given in Table~\ref{tab3}.}	\label{tab4}
	\end{longtable}
	\end{cor}
	\subsection{Spectra of hypergraphs constructed by unary hypergraph operations}\label{sec 4.2}
	\begin{notations}\normalfont
		Let $X=\{2,3,\dots, nl\}$ and $r\in\{1,2,\dots, n\}$, where $l,k\in \mathbb{N}\backslash \{1\},~l\leq k;$ $n\in\mathbb{N}$.
		\begin{itemize}
			\item[(i)]
			For $c\in X$, let us denote
			\begin{equation}\label{p1}
				p_1^{(c)}=
				\begin{cases}
					\displaystyle\underset{\text{for some}~p(p\neq 1)}{\underset{\underset{t_j\geq0,~t_p>0 }{t_1+t_2+\cdots+ t_l=c-2,}}{\sum}}\left(^{n-2}_{~~t_1}\right)\left(^{n}_{t_2}\right)\cdots\left(^{n}_{t_{i}}\right) \cdots \left(^{n}_{t_l}\right) &\text{if}~c-2>0\\
					0  & \text{otherwise}.
				\end{cases}	
			\end{equation}
			and
			\begin{equation}\label{p2}
				p_2^{(c)}=\begin{cases}
					\displaystyle\underset{t_j\geq0}{\underset{t_1+t_2+\cdots+ t_l=c-2,}{\sum}}\left(^{n-1}_{~~t_1}\right)\left(^{n-1}_{~~t_2}\right)\left(^{n}_{t_{3}}\right)\left(^{n}_{t_{4}}\right) \cdots \left(^{n}_{t_{l}}\right)&\text{if}~ c-2\geq0\\
					0  & \text{otherwise}.
				\end{cases}
			\end{equation}
			
			\item[(iii)]
			Let $x_1:=\frac{1}{2r-1}\left(^{n-1}_{r-1}\right)$;\\			
			Let $x_2:= \begin{cases}
				0 &\text{if}~    
				r=1;  \\
				\frac{1}{2r-1}\left(^{n-2}_{r-2}\right)& \text{otherwise}.
			\end{cases} $\\
			
		\end{itemize}
	\end{notations}
	Let $H(V(H), E(H))$ be a hypergraph. Consider a weight function $W:E(H)\rightarrow \mathbb{R}_{\geq 0}$ defined by,	
	\begin{align}\label{Hwt}
		W(e)=w_{|e|} ~~~\text{for all}~e\in E(H).
	\end{align} 
	In the following theorem, we obtain the characteristic polynomial of the adjacency, the Laplacian, the normalized Laplacian matrices of the weighted hypergraphs given in S.Nos.$1$-$36$ of Table~\ref{utab} by assuming a weight function given in \eqref{Hwt} on each of the constituting hypergraphs.
	
	\begin{thm}\label{t4.2}
		Let $H$ be a hypergraph on $n$ vertices. Consider the hypergraphs $H, H^c, K_n, \overline{H}$ with the weight function given in \eqref{Hwt}. Let $G_1, G_2\in\{H, H^c, K_n, K_n^c, \overline{H}\}$. Let $\mathcal{V}=(V(G_i))_{i=1}^2$ and $T\in \{_rT[\mathcal{V}], I[\mathcal{V}], J[\mathcal{V}], \Im_r[\mathcal{V}], \Im[\mathcal{V}]\}$. If $H$ is $r'$-regular, then the characteristic polynomial of the adjacency (resp. the Laplacian, the normalized Laplacian) matrix of the weighted hypergraph $G_1 \underset{T}{\bigvee}G_2$ is  
		\begin{align*}
			\displaystyle\underset{t=1}{\overset{n}{\prod}} \left(x^2-x[\lambda_t(M(G_1)+\theta_1\beta I_n+\theta_1'\gamma J_n)+\lambda_t(M(G_2)+\theta_2\beta I_n+\theta_2'\gamma J_n)]\right.+\\	
			\left.[\lambda_t(M(G_1)+\theta_1\beta I_n+\theta_1'\gamma J_n) \times \lambda_t(M(G_2)+\theta_2\beta I_n+\theta_2'\gamma J_n)-(\lambda_{t}(\delta aI_n+\delta bJ_n))^2]\right),
		\end{align*}
		where for $i=1,2,~t=1,2,\dots,n$, $\lambda_t(M(G_i)+\theta_i\beta I_n+\theta_i'\gamma J_n)$ and $\lambda_{t}(\delta aI_n+\delta bJ_n)$ are the co-eigenvalues of the matrices $M(G_i)+\theta_i\beta I_n+\theta_i'\gamma J_n$ and $\delta aI_n+\delta bJ_n$, respectively and the values $\theta_i, \theta_i', \delta$ and $M(G_i)$ are given in Table~\ref{Utab};\\	
		$r_i=\begin{cases}
			r'&\text{if}~ G_i=H;\\
			m'-r' &\text{if}~ G_i=H^c;\\
			m-r' &\text{if}~G_i=\overline{H};\\
			m &\text{if}~ G_i=K_n;\\
			0 &\text{if}~ G_i=K_n^c,
		\end{cases}$ \\where 	$m'=\displaystyle\underset{i\in K}{\sum}w_i\left(^{n-1}_{i-1}\right)$, $K=\{|e|~|~e\in E(H)\}$ and $m=\underset{i=2}{\overset{n}{\sum}}w_i\left(^{n-1}_{i-1}\right)$. 
		\\
		For $i=1,2$, let $z_i=r_i+\beta+z;~z=n\gamma+a+nb$. The values $\beta, \gamma, a$ and $b$ are given in Table~\ref{tab2} corresponding to the tensor $T$ and the values $p_1^{(c)}$ and $p_2^{(c)}$ are given in~\eqref{p1} and~\eqref{p2}, respectively when $k=l=2$. 
		\setlength\extrarowheight{1pt}
		\renewcommand{\arraystretch}{1.0}
		\tabcolsep=0.11cm
		\begin{longtable}{ | M{4cm} | M{3.5cm} |M{1cm} |M{1.5cm} |M{2cm}| }
			\hline
			\textbf{Name of the matrix} &    $\theta_i$ & $\theta_i'$ & $\delta$ & $M(G_i)$ \\
			\hline	
			
			Adjacency matrix& $1$ & $1$ & $1$ & $A(G_i)$\\
			\hline	
			
			Laplacian matrix &  $\begin{cases}
				\frac{z}{\beta}& \text{if}~ \beta\neq0,\\ z&\text{if}~ \beta= 0.
			\end{cases}$
			&
			$-1$ & $-1$ & $L(G_i)$\\
			\hline	
			
			normalized Laplacian matrix&  
			$\begin{cases}
				\frac{1}{\beta}-\frac{1}{z_i}& \text{if}~ \beta\neq0,\\ -\frac{1}{z_i}&\text{if}~ \beta= 0.
			\end{cases}$
			& 	$-\frac{1}{z_i}$ & $-\frac{1}{\sqrt{z_1z_2}}$ & $-\frac{1}{z_i}A(G_i)$
			\\
			\hline	
			\caption{Necessary values for determining the spectrum of the matrices associated with the hypergraphs given in Table~\ref{utab}.}\label{Utab}
		\end{longtable}
		\setlength\extrarowheight{1pt}
		\renewcommand{\arraystretch}{1.0}
		\tabcolsep=0.11cm
		\begin{longtable}{ | M{2cm} | M{3.5cm} |M{3.5cm} |M{2.5cm} |M{3.5cm}| }
			\hline
			\textbf{Tensor} $T$  &    $\beta$ & $\gamma$ & $a$ & $b$ \\
			\hline	
			
			\textbf{$_rT[\mathcal{V}]$}  &    $-x_2\cdot w_{2r}$ & $x_2\cdot w_{2r}$ & $w_{2r}(x_1-x_2)$ & $x_2$ \\
			\hline	
			
			\textbf{$I[\mathcal{V}]$}  &    $0$ & $0$ & $1$ & $0$\\
			\hline	
			
			\textbf{$J[\mathcal{V}]$}  &    $-\underset{c=2}{\overset{2n}{\sum}}\frac{p_1^{(c)}\cdot w_c}{c-1}$ & $\underset{c=2}{\overset{2n}{\sum}}\frac{p_1^{(c)}\cdot w_c}{c-1}$ & $0$ & $\underset{c=2}{\overset{2n}{\sum}}\frac{p_2^{(c)}\cdot w_c}{c-1}$ \\
			\hline	
			
			\textbf{$\Im_r[\mathcal{V}]$}  &    $-\underset{c=2}{\overset{2n}{\sum}}\frac{p_1^{(c)}\cdot w_c}{c-1}+x_2\cdot w_{2r}$ & $\underset{c=2}{\overset{2n}{\sum}}\frac{p_1^{(c)}\cdot w_c}{c-1}-x_2\cdot w_{2r}$ & $w_{2r}(x_2-x_1)$ & $\underset{c=2}{\overset{2n}{\sum}}\frac{p_2^{(c)}\cdot w_c}{c-1}-x_2\cdot w_{2r}$ \\
			\hline	
			
			\textbf{$\Im[\mathcal{V}]$}  &    $-\underset{c=2}{\overset{2n}{\sum}}\frac{p_1^{(c)}\cdot w_c}{c-1}$ & $\underset{c=2}{\overset{2n}{\sum}}\frac{p_1^{(c)}\cdot w_c}{c-1}$ & $-x_1\cdot w_2$ & $\underset{c=2}{\overset{2n}{\sum}}\frac{p_2^{(c)}\cdot w_c}{c-1}-x_2\cdot w_2$ \\
			\hline
			\caption{The values of $\beta, \gamma, a$ and $b$ corresponding to the indicating tensor $T$.}\label{tab2}
		\end{longtable}
	\end{thm}
	\begin{proof}
		The adjacency (resp. the Laplacian, the normalized Laplacian) matrix of $G_1 \underset{T}{\bigvee}G_2$ is of the form, 
		\[\mathcal{A}=
		\begin{bmatrix}
			M(G_1)+\theta_1\beta I_n+\theta_1'\gamma J_n& \delta(aI_n+bJ_n)\\ \delta(aI_n+bJ_n)&M(G_2)+\theta_2\beta I_n+\theta_2'\gamma J_n  
		\end{bmatrix}_{2n\times2n}\]
		where for $i=1,2$ the values $\beta, \gamma, a$ and $b$ corresponding to the indicating tensor $T$ are given in Table~\ref{tab2} and $M(G_i),~ \theta_i,~ \theta_i',~ \delta$ are given in the statement of Theorem~\ref{t4.2}. Since, $G_i$s are regular hypergraphs, any pair of blocks of $\mathcal{A}$ commute with each other. Thus, the proof follows from Theorem~\ref{t1.3}.
	\end{proof}
	\begin{cor}\normalfont
		In Theorem~\ref{t4.2}, let $G_1=G_2$ (= $G$, say) be $r$-regular and let $\mu_1=c,
		\mu_2,\dots, \mu_n $ be the eigenvalues of $M(G)$. Then the characteristic polynomial of the adjacency (resp. the Laplacian, the normalized Laplacian) matrix of $G \underset{T}{\bigvee}G$ is  
		\begin{align*}
			(x^2-(2x-(c+\theta\beta+n\theta'\gamma))(c+\theta\beta+n\theta'\gamma)-\delta^2(a+nb)^2)\times\\ \displaystyle\underset{i=2}{\overset{n}{\prod}} \left(x^2-2(\mu_i+\theta\beta)x+(\mu_i+\theta\beta)^2-\delta^2a^2\right),
			%
		\end{align*}where 
	
		$c=\begin{cases}
			r& \text{for the characteristic polynomial of $A\left(G \underset{T}{\bigvee}G\right)$}; \\
			0 &\text{for the characteristic polynomial of $L\left(G \underset{T}{\bigvee}G\right)$ };\\
			-\frac{r}{z'}& \text{for the characteristic polynomial of $\mathcal{L}\left(G \underset{T}{\bigvee}G\right)$}.
		\end{cases}$ 
	
\noindent and	the values $\beta,\gamma,\delta,a,b, \theta(=\theta_1=\theta_2),\theta'(=\theta_1'=\theta_2'),r(=r_1=r_2),z'(=z_1'=z_2')$ are as given in Theorem~\ref{t4.2}.
	\end{cor}
	\begin{proof}
		From Theorem~\ref{t4.2}, the characteristic polynomial of the adjacency (resp. the Laplacian, the normalized Laplacian) matrix of the weighted hypergraph $G \underset{T}{\bigvee}G$ is  
		\begin{align}\label{eq}
			\displaystyle\underset{t=1}{\overset{n}{\prod}} \left(x^2-2x[\lambda_t(M(G)+\theta\beta I_n+\theta'\gamma J_n)]\right.+\left.(\lambda_t(M(G)+\theta\beta I_n+\theta'\gamma J_n))^2 -(\lambda_{t}(\delta aI_n+\delta bJ_n))^2\right),
		\end{align}
	where  $\theta(=\theta_1=\theta_2),\theta'(=\theta_1'=\theta_2'),r(=r_1=r_2),z'(=z_1'=z_2')$ are as given in Theorem~\ref{t4.2}
		Since $M(G)$ is a real symmetric matrix of order $n$ with the row sum $c$, there exists an orthogonal basis of $\mathbb{R}^{n}$ consisting of its eigenvectors, including the all-one vector $J_{n\times 1}$ corresponds to the eigenvalue $c$. Let us denote the eigenvectors of $M(G)$ by $X_{1}(=J_{n\times 1})
		,X_{2},\dots,X_{n}$ corresponding to the eigenvalues $\mu_{1}(=c), \mu_{2},\dots,\mu_{n}$. 
		
		Notice that,
		$ \lambda_1(M(G)+\theta\beta I_n+\theta'\gamma J_n)=c+\theta\beta+n\theta'\gamma$ and $\lambda_{1}(\delta aI_n+\delta bJ_n)=\delta (a+nb)$ are the co-eigenvalues corresponding to the common eigenvector $X_1$.
		
		For $i=2,\dots,n$, $\lambda_i(M(G)+\theta\beta I_n+\theta'\gamma J_n)=\mu_i+\theta\beta$ and $\lambda_{i}(\delta aI_n+\delta bJ_n)=\delta a$ are the co-eigenvalues corresponding to the common eigenvector $X_i$. Thus from equation~\eqref{eq}, we have
		\begin{align*}
			\displaystyle\left(x^2-2x(c+\theta\beta+n\theta'\gamma)+(c+\theta\beta+n\theta'\gamma)^2 -\delta^2 (a+nb)^2\right)\\\times\underset{i=2}{\overset{n}{\prod}} \left(x^2-2x(\mu_i+\theta\beta)+(\mu_i+\theta\beta)^2 -\delta^2 a^2\right).
		\end{align*}	
		This completes the proof.
	\end{proof}
	\begin{thm}
		Let $G_i(V_i,E_i,W_i), i=1,2,\dots,k$ be $k(>1)$ copies of a weighted $r'$-regular hypergraph $H$ with $|V(H)|=n$. Then the characteristic polynomial of the adjacency (resp. the Laplacian, the normalized Laplacian) matrix of the weighted hypergraphs given in Table~\ref{Tabb} is
		\begin{align}\label{eq4.1}
			&\displaystyle\underset{t=1}{\overset{n}{\prod}}\left[k\left(x-\lambda_t(\alpha A(H)+\beta I_n+\gamma J_n)+\lambda_t(aI_n+bJ_n)\right)^k\right. \nonumber\\		
			& \left.\qquad\qquad-\lambda_t(aI_n+bJ_n)(x-\lambda_t(\alpha A(H)+\beta I_n+\gamma J_n)+\lambda_t(aI_n+bJ_n))^{k-1}\right],
		\end{align}
		where the values $\alpha, \beta, \gamma, a, b$ are given in Table~\ref{tab} and for $t=1,2,\dots,n$, $\lambda_t(\alpha A(H)+\beta I_n+\gamma J_n)$, $\lambda_t(aI_n+bJ_n)$ are the co-eigenvalues of the matrices $\alpha A(H)+\beta I_n+\gamma J_n$, $aI_n+bJ_n$, respectively. Let $z=r'\alpha+\beta+n\gamma+(k-1)a+k(n-1)b$, where $\alpha, \beta, \gamma, a, b$ are taken corresponding to the matrix of the respective graphs given in Table~\ref{tab}.\\
		Also, for $X=\{2,3,\dots, ln\}$, $r\in\{1,2,\dots, n\}$, let
		
		 $p_1'=\displaystyle\left(^{k-1}_{l-1}\right)\underset{c\in X}{ \sum}\frac{w_c\cdot p_1^{(c)}}{c-1}$,~~~~ $p_2'=\displaystyle\left(^{k-2}_{l-2}\right)\underset{c\in X}{ \sum}\frac{w_c\cdot p_2^{(c)}}{c-1}$,  
		 
		 	$p_{21}=\frac{w_{lr}}{lr-1}\left(^{k-2}_{l-2}\right)\left(^{n-1}_{r-1}\right)$,~~~~	$ p_{22}=			
		\begin{cases}
			0 &\text{if}~    
			r=1; \\
			\frac{w_{lr}}{lr-1}\left(^{k-2}_{l-2}\right)\left(^{n-2}_{r-2}\right)& \text{otherwise},
		\end{cases}$ \\and~
		$ p_{12} =			
		\begin{cases}
			0 &\text{if}~    
			r=1;  \\
			\frac{w_{lr}}{lr-1}\left(^{k-1}_{l-1}\right)\left(^{n-2}_{r-2}\right)& \text{otherwise}.
		\end{cases} $	
		\setlength\extrarowheight{1pt}
		\renewcommand{\arraystretch}{1.0}
		\tabcolsep=0.11cm
		\begin{longtable}{ | M{4.5cm} | M{3.9cm} |M{1cm} |M{1.9cm} |M{1.5cm} |M{1.5cm} |M{1.5cm} | }
			\hline
			\textbf{Name of the hypergraph}  &  \textbf{~~Name of the matrix}  & $\alpha$ & $\beta$ & $\gamma$ & $a$ & $b$ \\
			\hline
			
			& Adjacency matrix & $1$ & $-p_{12}$ & $p_{12}$ & $p_{21}-p_{22}$ &$p_{22}$ \\	
			\cline{2- 7}
			
			$(l,r)$-mirror hypergraph of $H$ & Laplacian matrix & $-1$ & $z+p_{12}$ & $-p_{12}$ & $p_{22}-p_{21}$ & $-p_{22}$ \\
			\cline{2- 7}
			
			& normalized Laplacian matrix & $-\frac{1}{z}$ & $1+\frac{p_{12}}{z}$ & $-\frac{p_{12}}{z}$ & $\frac{p_{22}-p_{21}}{z}$ & $\frac{-p_{22}}{z}$ \\
			\hline
			
			%
			%
			
			& Adjacency matrix & $1$ & $-p_{1}'$ & $p_{1}'$ & $0$ & $p_2'$ \\
			\cline{2- 7}
			Join $l$-neighbourhood hypergraph of $H$  & Laplacian matrix & $-1$ & $z+p_1'$ & $-p_1'$ & $0$& $-p_2'$ \\
			\cline{2- 7}
			& normalized Laplacian matrix & $-\frac{1}{z}$ & $1+\frac{p_{1}'}{z}$ & $-\frac{p_1'}{z}$ & $0$ & $\frac{-p_2'}{z}$ \\
			\hline
			& Adjacency matrix  & $1$ & $p_{12}-p_{1}'$ & $p_{1}'-p_{12}$ & $p_{22}-p_{21}$ & $p_2'-p_{22}$ \\
			\cline{2- 7}
			$VC$-$(l,r)$-neighbourhood hypergraph of $H$& Laplacian matrix & $-1$ & $z+p_1'-p_{12}$ & $p_{12}-p_1'$ & $p_{21}-p_{22}$ & $p_{22}-p_2'$ \\
			\cline{2- 7}		
			& normalized Laplacian matrix & $-\frac{1}{z}$ & $1+\frac{p_{1}'-p_{12}}{z}$ & $\frac{p_{12}-p_1'}{z}$ & $\frac{p_{21}-p_{22}}{z}$ & $\frac{p_{22}-p_2'}{z}$ \\
			\hline
			\caption{Necessary values for determining the spectrum of the matrices associated with the hypergraphs given in Table~\ref{Tabb}}\label{tab}	\end{longtable}
		
	\end{thm}
	\begin{proof}
		The adjacency (resp. The Laplacian, The normalized Laplacian) matrix of the hypergraphs given in Table \ref{tab} is of the form
		\begin{center}
			$\mathcal{A}=I_k\otimes(\alpha A(H)+\beta I_n+\gamma J_n)+(J_k-I_k)\otimes (a I_n+b J_n)$
		\end{center}  with the values $\alpha, \beta, \gamma ,a,b$ corresponding to the hypergraphs as given in Table~\ref{tab}.
		Let
		\begin{eqnarray}
			\mathcal{D}&=&I_{k}\otimes [\lambda_t(\alpha A(H)+\beta I_n+\gamma J_n)-\lambda_t(aI_n+bJ_n)];\nonumber\\ M_t &=&\mathcal{D} + [\lambda_t(aI_n+bJ_n)\times J_{k\times 1}\times J_{1\times k}].\nonumber\\
			\text{By Theorem}~\ref{sylvester},~~~~~~~~~\nonumber\\
			P_{M_t}(x)&=&det(xI_k-\mathcal{D}-\lambda_t(aI_n+bJ_n)J_{k\times 1}J_{1\times k}) \nonumber\\
			&=&det(xI_k-\mathcal{D}) det(1-\lambda_t(aI_n+bJ_n)J_{1\times k}(xI_k-\mathcal{D})^{-1}J_{k\times 1})\nonumber \\
			&=& P_{\mathcal{D}}(x) det\left(1-\frac{\lambda_{t}(aI_n+bJ_n)\cdot k}{x-\lambda_t(\alpha A(H)+\beta I_n+\gamma J_n)+\lambda_t(aI_n+bJ_n)}\right),\nonumber
		\end{eqnarray}
		where $P_{\mathcal{D}}(x)=\displaystyle \left(x-\lambda_t(\alpha A(H)+\beta I_n+\gamma J_n)+\lambda_t(aI_n+bJ_n)\right)^k.$ 
		Therefore,
		\begin{eqnarray}
			P_{M_t}(x)&=&k(x-\lambda_t(\alpha A(H)+\beta I_n+\gamma J_n)+\lambda_t(aI_n+bJ_n))^k\nonumber\\
			& &	-\lambda_t(aI_n+bJ_n)(x-\lambda_t(\alpha A(H)+\beta I_n+\gamma J_n)+\lambda_t(aI_n+bJ_n))^{k-1}\nonumber.
		\end{eqnarray}
		Applying Theorem~\ref{t1.3}, we have $P_\mathcal{A}(x)=\displaystyle\underset{t=1}{\overset{n}{\prod}}P_{M_t}(x)$, as desired.
	\end{proof}
	\subsection{Spectra of the $(H,\mathcal{T})$-join of hypergraphs}
	Let $H$ be a hypergraph and let $\mathcal{G}=(G_i(V_i,E_i,W_i))_{i=1}^{k}$ be a sequence of weighted hypergraphs. Let $E$ be the edge set of the hypergraph $\mathcal{G}(H,\mathcal{T})$. We define a weight function $W:E\rightarrow\mathbb{R}_{\geq 0}$ as follows:
	\begin{align}\label{hwt}
		W(e')=\begin{cases}
			W_i(e')&\text{if}~ e'\in E_i;\\
			w_{|e|}&\text{if}~~ e'\in E(T[\mathcal{V}_e]),
		\end{cases}
	\end{align} where  $\mathcal{V}_e=(V_i)_{i\in e}$ for each $e\in E(H)$.
	We denote the hypergraph $\mathcal{G}(H,\mathcal{T})$ together with a weight function $W$ given in~\eqref{hwt} by $\mathcal{G}(H,\mathcal{T},W)$.
	
	Throughout this subsection, we consider a weight function as defined above for any $(H,\mathcal{T})$-join of hypergraphs in $\mathcal{G}$.
	\begin{thm}\label{t3.2}
		Let $H$ be a hypergraph on $k$ vertices. Let $\mathcal{G}=(G_i(V_i,E_i,W_i))_{i=1}^k$ be a sequence of $r_i$-regular weighted hypergraphs $G_i$ with $|V_i|=n_i$ and let $X=\{2,3,\dots,N\}.$ For each $e\in E(H)$, let $\mathcal{V}_e=(V_i)_{i\in e}$ and let $\mathcal{T}=\{T[\mathcal{V}_e]~|~e\in E(H)\}$ be such that for each $p\in V_i,~ q\in V_j$ and $c\in X$, $|E_{p,q}^c(T[\mathcal{V}_e])|$ is a constant, say $n_{ij}^{c}(e)$ for all $i,j\in e$ and $1\leq i\leq j\leq k$. 
		Then the characteristic polynomial of the adjacency (resp. the Laplacian, the normalized Laplacian) matrix of the weighted hypergraph
		$\mathcal{G}(H,\mathcal{T},W)$ is
		\begin{center}
			$\underset{i=1}{\overset{k}{\prod}}~\underset{j=1;j\neq i}{\overset{n_i}{\prod}}(x-\alpha_i\lambda_{ij}-\beta_i)\times P_R(x)$,
		\end{center} where $\lambda_{ij}$ is a non-Perron adjacency eigenvalue of $G_i$ for all $j=1,2,\dots,n_i,~ i=1,2,\dots,k$ and
			\begin{center}
			$R=\begin{bmatrix}
			r_1\alpha_1+\beta_1+n_1\gamma_1 & n_2\delta_{12}\Delta_{12} &\cdots & n_k	\delta_{1k}\Delta_{1k}\\
			n_1\delta_{12}\Delta_{12} &r_2\alpha_2+\beta_2+n_2\gamma_2  & \cdots & n_k\delta_{2k}\Delta_{2k}\\
			\vdots & \vdots & \ddots& \vdots  \\
			n_1	\delta_{1k}\Delta_{1k}& n_2\delta_{2k}\Delta_{2k} & \cdots &   r_k\alpha_k+\beta_k+n_k\gamma_k
			\end{bmatrix}$
		\end{center}
		where, $\Delta_{ij}=\displaystyle\underset{c\in X}{ \sum}~\displaystyle\underset{\underset{e\in E(H)}{i,j\in e,}}{\sum} \frac{w_c\cdot n_{ij}^{c}(e)}{c-1}$ and the values $\alpha_i, \beta_i, \gamma_i$, $\delta_{ij}$ can be computed using Table~\ref{values table1} by taking $n_{ij}^{(c)}=\displaystyle\underset{\underset{e\in E(H)}{i,j\in e,}}{\sum}{~~~n_{ij}^c(e)}$ for all $1\leq i\leq j\leq k$ in Theorem~\ref{t2.2}.
	\end{thm}
	\begin{proof}
		As in Theorem~\ref{r3.1}, $(H,\mathcal{T})$-join of hypergraphs in $\mathcal{G}$ can be viewed as a $T[\mathcal{V}]$-join of hypergraphs in $\mathcal{G}$ for some suitable indicating tensor $T[\mathcal{V}]$, where $\mathcal{V}=(V_i)_{i=1}^k$. Since, $p\in V_i,~ q\in V_j$ and $c\in X$, $|E_{p,q}^c(T[\mathcal{V}_e])|=n_{ij}^{c}(e)$, for all $i,j\in e$ and $1\leq i\leq j\leq k$, we have $n_{ij}^{(c)}$ is a constant and is equal to $\underset{\underset{e\in E(H)}{i,j\in e,}}{\sum}{~~~n_{ij}^c(e)}$. Thus the proof follows from Theorem~\ref{t2.2}.
	\end{proof}
	In the following corollary, we construct infinite families of cospectral hypergraphs by using the $(\mathcal{H},\mathcal{T})$-join operation on hypergraphs.
	\begin{cor}\normalfont
		Let $H$ be a hypergraph on $k$ vertices and let $G_i(V_i,E_i,W_i)$, $G_i'(V_i',E_i',W_i')$ be  $A-$cospectral $r_i$-regular weighted hypergraphs
		for $i=1,2,\dots,k$. Let $\mathcal{G}=(G_i)_{i=1}^k$, $\mathcal{G}'=(G_i')_{i=1}^k$ and $X=\{2,3,\dots,N\}$.
		For each $e\in E(H)$, let $\mathcal{V}_e=(V_i)_{i\in e}$, $\mathcal{V}'_e=(V_i')_{i\in e}$. Let $\mathcal{T}=\{T[\mathcal{V}_e]~|~e\in E(H)\}$ be such that, for each $p\in V_i,~ q\in V_j$ and $c\in X$, $|E_{p,q}^c(T[\mathcal{V}_e])|$ is a constant, say $n_{ij}^{c}(e)$, for all $i,j\in e$ and $1\leq i\leq j\leq k$. Let $\mathcal{T}'=\{T[{\mathcal{V}'_e}]~|~e\in E(H)\}$, where $T[{\mathcal{V}'_e}]=T[\mathcal{V}_e]$. Then the hypergraphs	$\mathcal{G}(H,\mathcal{T},W)$ and	$\mathcal{G}'(H,\mathcal{T}',W)$ are simultaneously $A-$cospectral, $L-$cospectral and $\mathcal{L}-$cospectral.
	\end{cor}
	\begin{proof}
		Since $G_i$ and $G_i'$ are $r_i$ regular and the values $\alpha_i$, $\beta_i$, $\gamma_i$, $\delta_i$ depend only upon the indicating tensor $T[\mathcal{V}_e]$, from Theorem~\ref{t3.2}, the matrix $R$ is the same for the hypergraphs $\mathcal{G}(H,\mathcal{T},W)$ and	$\mathcal{G}'(H,\mathcal{T}',W)$. Since $G_i$ and $G_i'$ have the same $A$-spectrum, the result follows.
	\end{proof}
Now we proceed to obtain various spectrum of the hypergraphs given in Table~\ref{Tab4} by viewing them as a $(H,\mathcal{T})$-join of hypergraphs. In the following corollary, we deduce some results on the spectra of hypergraphs in the literature.
	\begin{cor}(\cite[Theorems 3.2.1,  3.1.1]{sarkar2020joins})
		\begin{enumerate}
			\item[(i)] 	Let $H$ be a hypergraph with $|V(H)|=k$ and $\mathcal{G}=(G_i(V_i,E_i,W_i))_{i=1}^k$ be a sequence of $r_i$-regular weighted hypergraphs $G_i.$ Then the characteristic polynomial of the adjacency matrix of the weighted join of set $\mathcal{G}$ of non-uniform hypergraphs on a backbone hypergraph $H$ given in Table~\ref{Tab4} is obtained from Theorem~\ref{t2.2} by taking the values of $z_i, \alpha_i, \beta_i, \gamma_i$ and $\delta_{ij}$ as given in Theorem~\ref{t2.2} and taking \begin{align}\label{eqn 3.3}
				n_{ij}^{(c)}=\begin{cases}
					\displaystyle\underset{\underset{e\in E(H)}{i,j\in e}}{\sum}n_{ij}^{c}(S_e)& \text{if}~ c\in B_e;\\
					0& \text{otherwise.}
				\end{cases}
			\end{align}
			\item [(ii)]Let $H$ be a hypergraph with $|V(H)|=k$ and let $\mathcal{G}=(G_i(V_i,E_i,W_i))_{i=1}^k$ be a sequence of $r_i$-regular $m$-uniform weighted hypergraphs $G_i(V_i,E_i,W_i).$		
			Then the characteristic polynomial of the adjacency matrix of the weighted join of set $\mathcal{G}$ of $m$-uniform hypergraphs on a backbone hypergraph $\mathcal{H}$ given in Table~\ref{Tab4} is obtained from Corollary~\ref{c2.1} by taking the values of $z_i, \alpha_i, \beta_i, \gamma_i$ and $\delta_{ij}$ as in Corollary~\ref{c2.1} and the value of $n_{ij}^{(c)}$ as given in~\eqref{eqn 3.3} with $c=m$.
		\end{enumerate}	
	\end{cor}
	\begin{cor}
		The characteristic polynomial of the Laplacian matrix and the normalized Laplacian matrix of the weighted join of set $\mathcal{G}$ of weighted non-uniform hypergraphs on a backbone hypergraph $\mathcal{H}$ given in S.No.$1$ of Table~\ref{Tab4} with the weight function given in~\eqref{hwt} can be obtained from Theorem~$\ref{t2.2}$ by taking the values $z_i, \alpha_i, \beta_i, \gamma_i$ and $\delta_{ij}$ corresponds to the Laplacian, the normalized Laplacian matrices given in Theorem~\ref{t2.2} and the value $n_{ij}^{(c)}$ as given in~\eqref{eqn 3.3}.
	\end{cor}
	\begin{cor}
		The characteristic polynomial of the Laplacian matrix, the normalized Laplacian matrix of the weighted join of set $\mathcal{G}$ of weighted $m$-uniform hypergraphs on a backbone hypergraph $\mathcal{H}$ given in S.No.$2$ of Table~\ref{Tab4} with the weight function given in~\eqref{hwt} can be obtained from Corollary~\ref{c2.1} by taking the values $z_i, \alpha_i, \beta_i, \gamma_i$ and $\delta_{ij}$ corresponds to the Laplacian, the normalized Laplacian matrices given in Corollary~\ref{c2.1} and the value $n_{ij}^{(c)}$ as given in~\eqref{eqn 3.3} with $c=m$.
	\end{cor}
	\begin{notation}
		For $1\leq i\leq j\leq k$, let \\
		$q_{ij}^{(c)}=\begin{cases}
			\displaystyle\underset{\underset{|\{p_1,p_2,\dots,p_{c-2}\}|=c-2}{\{p_1,p_2,\dots,p_{c-2}\}\subseteq \{1,2,\dots,k\}\backslash\{i,j\},}}{\sum}{n_{p_1}n_{p_2}\dots n_{p_{c-2}}}& \text{if}~~ i\neq j;\\
			0& \text{otherwise}.
		\end{cases}$
	\end{notation}
	\begin{cor}\label{c3.4}\normalfont
		\begin{itemize}
			\item[(i)] 	The characteristic polynomials of the adjacency, the Laplacian, the normalized Laplacian matrices of the weighted complete $m$-uniform strong $k$-partite hypergraph and weighted complete strong $k$-partite hypergraph mentioned in Table~\ref{Tab4} 
			are derived from Theorem~\ref{t2.2} by using the necessary values given in Table~\ref{H gen table} and $\alpha_i$, $\beta_i$, $\gamma_i$, $\delta_i$ are taken as given in Theorem~\ref{t2.2}. 
			\item[(ii)]  The characteristic polynomials of the adjacency, the Laplacian, the normalized Laplacian matrices of the weighted lexicographic product of a hypergraph $H$ and a $r$-regular weighted hypergraph $H'$ mentioned in Table~\ref{Tab4}
			are obtained from Theorem~\ref{t3.2} by taking the values given in S.No.$3$ of Table~\ref{H gen table}.
		\end{itemize}
		\setlength\extrarowheight{1pt}
		\renewcommand{\arraystretch}{1.0}
		\tabcolsep=0.11cm
		\begin{longtable}{ | M{1.3cm} | M{5.5cm} |M{6.5cm}| }
			\hline
			\textbf{S. No.}  & \textbf{Name of the hypergraph}  &
			\textbf{Values} \\ 
			\hline
			
			1. & Complete $m$-uniform strong $k$-partite hypergraph  & $X=\{m\};~~ n_{ij}^{(m)}=q_{ij}^{(m)}$  \\
			\hline
			
			2. & Complete strong $k$-partite hypergraph  & $X=\{2,3,\dots, k\};~~n_{ij}^{(c)}=q_{ij}^{(c)}$  \\
			\hline
			
			3. & Lexicographic product of $H(V,E)$ and $H'(V',E')$ & $r_i=r;~n_i=n;~\alpha_i=1,~\beta_i= \gamma_i=0;$ $X=\{|e|~|e\in E\};$ $n_{ij}^{c}(e)=
			\begin{cases}
				|V'|^{|e|-2}~~\text{if}~ i\neq j\\
				~~0~~~~\text{otherwise.}
			\end{cases}$  \\
			\hline
			\caption{Necessary values to compute the spectrum of the hypergraphs given in Table~\ref{Tab4}}\label{H gen table}
		\end{longtable}
	\end{cor}
	\begin{cor}
		If $H, H'$ are $k$-uniform hypergraphs with $|V|=n$, $|V'|=m$ and if $H'$ is weighted $r$-regular, then the characteristic polynomial of the adjacency matrix of the weighted lexicographic product of $H$ and $H'$ is 
		\begin{equation*}
			\displaystyle\underset{\lambda}{\prod}(x-\lambda)^{n}\underset{\mu}{\prod}(x-r-\mu m^{k-1}w_k),
		\end{equation*} where the products run over all the non-Perron eigenvalues $\lambda$ of $A(H')$ and all the eigenvalues $\mu$ of $A(H)$ respectively. The weight function considered in this lexicographic product is as given in~\eqref{hwt}.
	\end{cor}
	\begin{proof} The lexicographic product of $H$ and $H'$ can be viewed as a $(H,\mathcal{T})$-join of hypergraphs as mentioned in Table~\ref{Tab4}.
		So we take $G_i=H'$, $n_i=m$, $r_i=r$, $\alpha_i=1$, $\beta_i=\gamma_i=0$, $\delta_{ij}=1$ and $n_{ij}^k(e)=m^{k-2}$ for all $1\leq i\leq j\leq k$ in Theorem~\ref{t3.2}. Then the matrix $R$ becomes $rI_n+(w_km^{k-1})A(H)$. Since it is a polynomial in $A(H)$, the proof follows.
	\end{proof}

	\section*{Acknowledgement} The first author thanks University Grants Commission (UGC), Government of India for the financial support in the form of Junior Research Fellowship (NTA Ref. No.: 221610053976).


\begin{thebibliography}{99}
	
	\bibitem{agarwal2006higher}
	S. Agarwal, K. Branson and S. Belongie.
	\newblock Higher order learning with graphs,
	\newblock In {\em Proceedings of the 23rd International Conference on Machine
		Learning}, Pittsburgh,  Pennsylvania, USA, June 25-29, 2006, pp. 17--24. \url{DOI: 10.1145/1143844.1143847}
	
	\bibitem{banerjee2021spectrum}
	A. Banerjee,
	\newblock On the spectrum of hypergraphs,
	\newblock {\em Linear Algebra Appl.}, {\bf614} (2021), 82--110. 
	
	\bibitem{banerjee2017spectra}
	A. Banerjee, A. Char and B. Mondal,
	\newblock Spectra of general hypergraphs,
	\newblock {\em Linear Algebra Appl.}, {\bf518} (2017), 14--30. 
	
	\bibitem{cooper2012spectra}
	J. Cooper and A. Dutle,
	\newblock Spectra of uniform hypergraphs,
	\newblock {\em Linear Algebra Appl.}, {\bf436} (2012), 3268--3292. 
	
	\bibitem{cvetkovic2010introduction}
	D. Cvetkovi{\'c}, P. Rowlinson and S. Simi{\'c},
	\newblock {\em An Introduction to the Theory of Graph Spectra}, volume~{\bf75},
	\newblock Cambridge University Press, Cambridge, 2010.
	
	\bibitem{gayathrithesis}
	M. Gayathri,
	\newblock {\em Spectra of graphs consttruced by various new graph operations},
	\newblock Ph.D. Thesis, The Gandhigram Rural Institute (Deemed to be  University), 2020.
	
	\bibitem{gayathri2019adjacency}
	M. Gayathri and R. Rajkumar,
	\newblock Adjacency and laplacian spectra of variants of neighborhood corona of
	graphs constrained by vertex subsets,
	\newblock {\em Discrete Math. Algorithms Appl.},
	{\bf11} (2019), Article No. 1950073. 
	
	\bibitem{murugesan2021spectra}
	M. Gayathri and R. Rajkumar,
	\newblock Spectra of partitioned matrices and the $\mathcal{M}$-join of graphs,
	\newblock {\em Ric. Mat.},  (2021),  1--48, (Published online) \url{https://doi.org/10.1007/s11587-021-00589-x}
	
	\bibitem{gosselin2012self}
	S. Gosselin,
	\newblock Self-complementary non-uniform hypergraphs,
	\newblock {\em Graphs Combin.}, {\bf28} (2012), 615--635. 
	
	\bibitem{haynsworth1960reduction}
	E. V. Haynsworth,
	\newblock A reduction formula for partitioned matrices,
	\newblock {\em J. Res. Nat. Bureau Stand.}, {\bf64} (1960), 171--174.
	
	\bibitem{hedetniemi1969classes}
	S. Hedetniemi,
	\newblock On classes of graphs defined by special cutsets of lines,
	\newblock In: {\em The many facets of graph theory}, Springer,
	1969, 171--189. 
	
	\bibitem{hellmuth2012survey}
	M. Hellmuth, L. Ostermeier and P.F. Stadler,
	\newblock A survey on hypergraph products,
	\newblock {\em Math. Comput. Sci.}, {\bf 6} (2012), 1--32. 
	
	\bibitem{meyer2000matrix}
	C.D. Meyer,
	\newblock {\em Matrix Analysis and Applied Linear Algebra}, volume~{\bf 71},
	\newblock SIAM, Philadelphia, 2000.
	
	
	\bibitem{pavithra2021bowtie}
	R. Pavithra and R. Rajkumar,
	\newblock Spectra of bowtie product of graphs,
	\newblock {\em Discrete Math. Algorithms Appl.}, {\bf14} (2021), Article No. 2150114. 
	
	\bibitem{pavithra2021spectra}
	R. Pavithra and R. Rajkumar,
	\newblock Spectra of $M$-edge rooted product of graphs,
	\newblock {\em Indian J. Pure Appl. Math.}, {\bf52} (2021), 1235--1255. 
	
	
	\bibitem{pearson2014spectral}
	K. J. Pearson and T. Zhang,
	\newblock On spectral hypergraph theory of the adjacency tensor,
	\newblock {\em Graphs Combin.}, {\bf30} (2014), 1233--1248. 
	
	\bibitem{qi2017tensor}
	L. Qi and Z. Luo,
	\newblock {\em Tensor Analysis: Spectral Theory and Special Tensors},
	\newblock SIAM, Philadelphia, 2017.
	
	\bibitem{rajkumar2019spectra}
	R. Rajkumar and M. Gayathri.
	\newblock Spectra of $(H_1, H_2)$-merged subdivision graph of a graph,
	\newblock {\em Indag. Math.}, {\bf30} (2019), 1061--1076. 
	
	\bibitem{rajkumar2021spectra}
	R. Rajkumar and M. Gayathri,
	\newblock Spectra of generalized corona of graphs constrained by vertex
	subsets,
	\newblock {\em Matematiche}, {\bf76} (2021), 211--241.
	
	\bibitem{rajkumar2020spectra}
	R. Rajkumar and R. Pavithra,
	\newblock Spectra of $M$-rooted product of graphs,
	\newblock {\em Linear Multilinear Algebra}, {\bf70} (2020), 1--26. 
	
	\bibitem{rodri2002laplacian}
	J. A. Rodriguez,
	\newblock On the Laplacian eigenvalues and metric parameters of hypergraphs,
	\newblock {\em Linear Multilinear Algebra}, {\bf50} (2002), 1--14. 
	
	\bibitem{sarkar2020joins}
	A. Sarkar and A. Banerjee,
	\newblock Joins of hypergraphs and their spectra,
	\newblock {\em Linear Algebra Appl.}, {\bf603} (2020), 101--129. 
	
	\bibitem{voloshin2009introduction}
	V.I. Voloshin,
	\newblock {\em Introduction to Graph and Hypergraph Theory},
	\newblock Nova Science Publishers, Hauppauge, 2009.
\end{thebibliography}

\end{document}